\documentclass{amsart}

\usepackage{amsfonts}
\usepackage{amssymb}
\usepackage{mathrsfs}
\usepackage{amsmath}
\usepackage{fullpage}
\usepackage{paralist}
\usepackage{hyperref}
\newcommand{\arxiv}[1]{\href{http://www.arXiv.org/abs/#1}{arXiv:#1}}
\usepackage{tikz}
\usepackage{graphicx,graphics}

\newtheorem{theorem}{Theorem}[section]

\newtheorem{corollary}[theorem]{Corollary}

\newtheorem{definition}[theorem]{Definition}
\newtheorem{example}[theorem]{Example}

\newtheorem{lemma}[theorem]{Lemma}

\newtheorem{proposition}[theorem]{Proposition}
\newtheorem{remark}[theorem]{Remark}

\DeclareMathAlphabet{\mathbbold}{U}{bbold}{m}{n}

\def\supp{\operatorname{supp}}
\def\bzero{{\mathbbold 0}}
\def\bunity{{\mathbbold 1}}
\def\mysup{\vee}

\def\R{{\mathbb R}}
\def\Z{{\mathbb Z}}
\def\N{{\mathbb N}}
\def\T{{\mathbb T}}
\def\Rmax{\T}
\def\Rmaxplus{\R_{\max,+}}
\def\Rmaxtimes{\R_{\max,\times}}
\def\Rinv{\T_{+}}

\def\Rpn{\R_+^n}

\def\rec{\operatorname{rec}}
\def\ConicSec{{\cW}}
\def\CompConicSec{{\complement \ConicSec}}
\def\Sector{{\cS}}
\def\CompSector{{\complement \Sector}}

\def\cH{{\mathcal H}}
\def\cC{{\mathcal C}}
\def\cD{{\mathcal D}}
\def\cV{{\mathcal V}}
\def\cW{{\mathcal W}}
\def\cS{{\mathcal S}}

\def\cone{\operatorname{span}}
\def\conv{\operatorname{conv}}

\def\mysup{\bigoplus}
\def\bez{\backslash}

\newcommand{\uvector}[1]{{\rm e}^{#1}}

\catcode`\@=11
\@namedef{subjclassname@2010}{%
  \textup{2010} Mathematics Subject Classification}
\catcode`\@=12

\begin{document}

\title{Characterization of tropical hemispaces by $(P,R)$-decompositions}

\author{Ricardo D. Katz}

\address{CONICET. Postal address: Instituto de Matem\'atica
``Beppo Levi'', Universidad Nacional de Rosario,
Avenida Pellegrini 250, 2000 Rosario, Argentina.}
\email{rkatz@fceia.unr.edu.ar}

\author{Viorel Nitica}

\address{Department of Mathematics, West Chester University, PA
19383, USA, and Institute of Mathematics, P.O. Box 1-764, Bucharest, Romania.}
\email{vnitica@wcupa.edu}

\author{Serge\u{\i} Sergeev}

\address{University of Birmingham,
School of Mathematics, Watson Building, Edgbaston B15 2TT, UK.}
\email{sergiej@gmail.com}

\thanks{R. D. Katz is partially supported by CONICET Grant PIP 112-201101-01026.
V. Nitica is supported by Simons Foundation Grant 208729. S. Sergeev
is supported by EPSRC grant EP/J00829X/1 and (partially) by
RFBR-CNRS grant 11-0193106 and RFBR grant 12-01-00886.}


\begin{abstract}
We consider tropical hemispaces, defined as tropically convex sets
whose complements are also tropically convex, and tropical
semispaces, defined as maximal tropically convex sets not containing
a given point. We introduce the concept of $(P,R)$-decomposition.
This yields (to our knowledge) a new kind of representation of tropically
convex sets extending the classical idea of representing convex sets
by means of extreme points and rays. We characterize
tropical hemispaces as tropically convex sets that admit a
$(P,R)$-decomposition of certain kind. In this characterization,
with each tropical hemispace we associate a matrix with coefficients
in the completed tropical semifield, satisfying an extended rank-one
condition. Our proof techniques are based on homogenization (lifting
a convex set to a cone), and the relation between tropical
hemispaces and semispaces.
\end{abstract}

\maketitle

\section{Introduction}

{\em Max-plus algebra} is the algebraic structure
obtained when considering the {\em max-plus semifield} $\Rmaxplus$.
This semifield is defined as the set $\R\cup \{-\infty\}$
endowed with $\alpha \oplus \beta:= \max(\alpha, \beta)$
as addition and the usual real numbers addition
$\alpha \otimes \beta:=\alpha +\beta$ as multiplication. Thus,
in the max-plus semifield,
the neutral elements for addition and
multiplication are $-\infty$ and $0$ respectively.

The max-plus semifield is algebraically isomorphic to the
{\em max-times semifield} $\Rmaxtimes$,
also known as the max-prod semifield (see e.g.~\cite{NS-07I,NS-07II}),
which is given by the set $\R_+=[0,+\infty)$
endowed with $\alpha\oplus \beta:= \max(\alpha, \beta)$
as addition and the usual real numbers product
$\alpha\otimes \beta:=\alpha \beta$ as multiplication. Consequently,
in the max-times semifield,
$0$ is the neutral element for addition and
$1$ is the neutral element for multiplication.

In this paper we consider both of these semifields at
the same time, under the common notation $\Rmax$ and under the common name {\em tropical algebra}.
In what follows $\Rmax$ denotes either the max-plus semifield $\Rmaxplus$ or the max-times semifield $\Rmaxtimes$.
We will use $\bzero$ to denote the neutral element for addition, $\bunity $ to denote the neutral element
for multiplication, and $\Rinv$ to denote the set of all
invertible elements with respect to the multiplication,
i.e., all the elements of $\Rmax$ different from $\bzero$.

The space $\Rmax^n$ of $n$-dimensional vectors $x=(x_1,\dots,x_n)$,
endowed naturally with the component-wise addition
(also denoted by $\oplus$) and $\lambda x:=(\lambda \otimes x_1,\dots,\lambda \otimes x_n)$
as the multiplication of a scalar $\lambda\in\Rmax$ by a vector $x$,
is a semimodule over $\Rmax$.
The vector $(\bzero, \dots, \bzero)\in \Rmax^n$ is also denoted by $\bzero$,
and it is the identity for $\oplus$.

In {\em tropical convexity}, one first defines the
{\em tropical segment} joining the points $x,y\in\Rmax^n$ as the set $\left\{\alpha x\oplus\beta y\in\Rmax^n\mid
\alpha,\beta\in\Rmax,\,\alpha\oplus\beta= \bunity\right\}$,
and then calls a set $\cC \subseteq \Rmax^n$ {\em tropically convex}
if it contains the tropical segment
joining any two of its points
(see Figure~\ref{f:segment-plane} below for an illustration of
tropical segments in dimension $2$).
Similarly, the notions of {\em cone, halfspace, semispace, hemispace, convex hull, linear span, convex and linear combination},
can be transferred to the tropical setting (precise definitions are given below).
Henceforth
all these terms used without precisions should always be understood in the max-plus or max-times (i.e. tropical) sense.

The interest in this convexity (also known as {\em max-plus convexity} when $\T=\Rmaxplus$, or
{\em max-times convexity} or {\em $\mathbb{B}$-convexity} when $\T=\Rmaxtimes $) comes from several fields,
some of which we next review.
Convexity in $\Rmax^n$ and in more general semimodules
was introduced by Zimmermann~\cite{Zim-77}
under the name ``extremal convexity'' with applications e.g. to discrete optimization problems and it was studied by
Maslov, Kolokoltsov, Litvinov, Shpiz and others as part of the
Idempotent Analysis~\cite{KM:97,LMS-01,MS:92},
inspired by the fact that the
solutions of a Hamilton-Jacobi equation associated with a deterministic
optimal control problem belong to structures similar to convex cones.
Another motivation arises from the algebraic approach to discrete event
systems initiated by Cohen et al.~\cite{CDQV-85},
since the reachable and observable
spaces of certain timed discrete event systems are naturally
equipped with structures of cones of $\Rmax^n$ (see e.g. Cohen et al.~\cite{cgq-1999}).
Motivated by tropical algebraic geometry and applications in phylogenetic analysis,
Develin and Sturmfels studied polyhedral convex sets in $\Rmax^n$
thinking of them as classical
polyhedral complexes~\cite{DS-04}.

Many results that are part of classical convexity
theory can be carried over to the setting of $\Rmax^n$:
separation of convex sets and projection operators (Gaubert and Sergeev~\cite{GS-08}),
minimization of distance and description of sets of best
approximation (Akian et al.~\cite{AGNS-11}), discrete convexity
results such as Minkowski theorem (Gaubert and Katz~\cite{GK-06,GK-07}),
Helly, Caratheodory and Radon theorems (Briec and Horvath~\cite{BH-04}), colorful Caratheodory
and Tverberg theorems (Gaubert and Meunier~\cite{GM-09}), to quote a few.

Here we investigate {\em hemispaces in $\Rmax^n$}, which are
convex sets in $\Rmax^n$ whose complements in $\Rmax^n$ are also convex.
The definition of hemispaces makes sense in other structures once the notion of convex set is defined. Hemispaces also appear in the
literature under the name of halfspaces, convex halfspaces,
and generalized halfspaces.
As general convex sets are quite complicated in many convexity structures,
a simple description of hemispaces is highly desirable.
Usual hemispaces in $\R^n$ are described by Lassak in~\cite{Lassak-84}.
Mart\'{\i}nez-Legaz and Singer~\cite{MLegSin-84}
give several geometric characterization of usual hemispaces in $\R^n$
with the aid of linear
operators and lexicographic order in $\R^n$.

Hemispaces play a role in abstract convexity (see Singer~\cite{Sin:97}, Van de Vel~\cite{VdV}),
where they are used in the Kakutani Theorem to separate
two convex sets from each other.
The proof of Kakutani Theorem makes use of
Zorn's Lemma (relying on the Pasch axiom,
which holds both in tropical
and usual convexity).
A different approach is to start from the
separation of a point from a closed convex set,
as investigated in many works (e.g., Zimmermann~\cite{Zim-77},
Litvinov~et~al.~\cite{LMS-01},
Cohen~et~al.~\cite{CGQ-04,CGQS-05},
Develin and Sturmfels~\cite{DS-04}, Briec~et~al.~\cite{BHR-05}).
This Hahn-Banach type result is extended to the
separation of several convex sets by an application of
non-linear Perron-Frobenius theory by Gaubert and Sergeev in~\cite{GS-08}.

In the Hahn-Banach approach, tropically convex sets
are separated by means of closed halfspaces in $\Rmax^n$,
defined as sets of vectors
$x$ in $\Rmax^n$ satisfying an inequality of the form
$\mysup_j \gamma_jx_j\oplus \alpha\leq \mysup_i \beta_ix_i\oplus\delta$.
As shown by Joswig~\cite{Jos-05},
closed halfspaces in $\Rmax^n$
are unions of several closed sectors, which are convex tropically and in the ordinary sense.

Briec and Horvath~\cite{BH-08} proved that the topological
closure of any  hemispace in $\Rmax^n$
is a closed halfspace in $\Rmax^n$.
Hence closed halfspaces,
with respect to general hemispaces,
are ``almost everything''. However,
the borderline between a hemispace
and its complement in $\Rmax^n$ has a generally unknown intricate pattern,
with some pieces belonging to
one hemispace and the rest to the other.
This pattern was not
revealed by Briec and Horvath.

The present paper gives a complete characterization of hemispaces in
$\Rmax^n$ by means of the so-called $(P,R)$-decompositions (see
Definition~\ref{def:PR} below). In dimension 2 the borderline is described
explicitly and all the types of hemispaces in $\Rmax^2$ that may
appear are shown in Figures~\ref{f:hemispace-plane}
and~\ref{f:hemispace-line}. Thus, our result is more general than
the one established in~\cite{BH-08} even in dimension 2. In higher
dimensions one may use the characterization in terms of
$(P,R)$-decompositions to describe the thin structure of the
borderline quite explicitly.

We now describe the basic idea of the proof of this characterization.
Let us first recall that like in usual convexity, a closed convex set in
$\Rmax^n$ can be decomposed as the (tropical) Minkowski sum of the convex hull of its extreme points and its recession cone (Gaubert and Katz~\cite{GK-06,GK-07}). As a relaxation of this traditional
approach, we suggest the concept of $(P,R)$-decomposition to
describe general convex sets in $\Rmax^n$.
Developed here in the context of tropical convexity,
this concept corresponds to that of Motzkin decomposition studied in usual convexity in locally convex spaces
(see e.g.~\cite{GGMT-10}).
Homogenization,
which carries convex sets to convex cones, is another classical tool
we exploit in the setting of $\Rmax^n$. Next, an important feature of
tropical convexity (as opposed to usual convexity) is the existence
of a finite number of types of semispaces, i.e., maximal
convex sets in $\Rmax^n$ not containing a given point. These sets were described
in detail by Nitica and Singer~\cite{NS-07I,NS-07II,NS-07}, who
showed that they are precisely the complements of closed sectors.
Let us mention that the multiorder principle of tropical
convexity~\cite{NS-07I,NS-07II,Ser-09-inLS,multi-order-1991}
can be formulated in terms of complements of semispaces.

It follows from abstract convexity that any hemispace
is the union of all the complements of semispaces which it contains.
These sets are closed sectors of several types.
The convex hull in $\Rmax^n$ of a union of sectors
of certain type gives a sector of the same type,
perhaps with some pieces of the boundary missing.
Some degenerate cases may also appear.
Sectors admit a (relatively) simple $(P,R)$-decomposition,
and we can combine such $(P,R)$-decompositions to obtain a
$(P,R)$-decomposition of the hemispace.
So far the method is quite general and geometric,
and in dimension $2$ sufficient for classification.

For higher dimensions the fact that we deal with hemispaces becomes
relevant. It turns out that a hemispace in $\Rmax^n$  admits a
$(P,R)$-decomposition consisting of unit vectors and linear
combinations of two unit vectors. Thus, to characterize a hemispace
by means of $(P,R)$-decompositions we need to understand how the
linear combinations of two unit vectors are distributed
among the hemispace and its complement. The proof becomes more
algebraic and combinatorial, and at this point it becomes convenient
to work with cones and their (usual) representation in
terms of generators. Using homogenization, we reduce the study
of general hemispaces in $\Rmax^n$ to the study of conical
hemispaces in $\Rmax^{n+1}$ (these are hemispaces in $\Rmax^{n+1}$
which are also cones or, equivalently, cones in
$\Rmax^{n+1}$ whose complements enlarged with $\bzero$ are also cones).
We introduce the ``$\alpha$-matrix'', whose entries
stem from the borderline between a conical hemispace and its
complement in two-dimensional coordinate planes. We show that it
satisfies an extended rank-one condition, and then we prove that
this condition is also sufficient in order for a set to
generate a conical hemispace. This part of the proof is more
technical and it is given in the last third of the paper, starting
with Proposition~\ref{p:RD} and ending with the proof of
Theorem~\ref{t:Ghemi}. We use the rank-one condition to describe the
fine structure of the $\alpha$-matrix, which is an independent
combinatorial result of interest, and then use this structure to
construct explicitly the complementary conical hemispace for a
conical hemispace given by its $(P,R)$-decomposition. Finally, we translate
this result back to the $(P,R)$-decomposition of general hemispaces,
to obtain the main result of the paper (Theorem~\ref{t:mainres}).

The paper is organized as follows. Section~\ref{s:prel} is occupied
with preliminaries on convex sets in $\Rmax^n$, and introduces the concept of
$(P,R)$-decomposition. In Section~\ref{s:semi} we study semispaces
in $\Rmax^n$, in order to give, exploiting homogenization, a simpler
proof of their characterization than the one given in~\cite{NS-07I,NS-07II}.
Hemispaces appear here as unions of (in general, infinitely many)
complements of semispaces, i.e., the closed sectors of~\cite{Jos-05}.
Section~\ref{s:hemi} contains the main
results on hemispaces in $\Rmax^n$. The purpose of
Subsection~\ref{ss:hemihomog} is to reduce general hemispaces in
$\Rmax^n$ to conical hemispaces in $\Rmax^{n+1}$.
This aim is finally achieved in Theorem~\ref{t:hemihomog}. In view
of this theorem, in Subsection~\ref{ss:hemicone} we
study conical hemispaces only. There we prove Theorem~\ref{t:Ghemi}
as explained above, which gives a concise characterization of
conical hemispaces in terms of generators. In
Subsection~\ref{ss:final}, we obtain a number of corollaries of the
previous results. In the first place we verify that closed
hemispaces in $\Rmax^n$ are closed halfspaces in $\Rmax^n$, a result
of~\cite{BH-08}, see Theorem~\ref{t:hemishalfs} and
Corollary~\ref{c:hemishalfs}. Finally, the main result of this paper
is given in Theorem~\ref{t:mainres} of Subsection~\ref{ss:mainres}.
It provides a characterization of general hemispaces in $\Rmax^n$ as
convex sets having particular $(P,R)$-decompositions, and
is obtained as a combination of Theorems~\ref{t:hemihomog}
and~\ref{t:Ghemi}.

\section{Preliminaries}\label{s:prel}

In the sequel, for any $m,n\in \Z$ with $m \leq n$,
we denote the set $\{m,m+1,\dots,n\}$ by $[m,n]$, or simply by $[n]$ when $m=1$.
The multiplicative inverse of $\lambda \in \Rinv$ (recall that $\Rinv:=\Rmax\backslash\{\bzero\}$) will be denoted
by $\lambda^{-1}$.
For $x\in \Rmax^n$ we define the {\em support} of $x$ by
\[
\supp(x):=\left\{i\in [n]\mid x_i\neq \bzero \right\} \; .
\]
We will say that $x\in \Rmax^n$ has \emph{full support} if
$\supp(x)=[n]$. Otherwise we say that $x$ has \emph{non-full
support}.

The set of the vectors $\{\uvector{i,n}\mid i\in [n]\}\subseteq \Rmax^n$
defined by
\begin{equation*}
\uvector{i,n}_j=\left \{ \begin{gathered}
\bunity \text{ if } i=j \\
\bzero \text{ if } i\not = j
\end{gathered}
\right .
\end{equation*}
form the {\em standard basis} in $\Rmax^n$.
We will refer to these vectors as the {\em unit vectors}.
In what follows, we will work with unit vectors in both $\Rmax^n$ and $\Rmax^{n+1}$. For simplicity of the notation, we identify $\uvector{i,n}$ with $\uvector{i,n+1}$ for $i\leq n$, and write simply $\uvector{i}$ for them.

To introduce a topology we need to specialize $\Rmax$ to one of the models.
Namely, if $\Rmax=\Rmaxtimes$ then we use the topology induced in $\Rpn$
by the usual Euclidean topology in the real space.
If $\Rmax=\Rmaxplus$, then our topology is induced by the metric
$d_{\infty}(x,y)=\max_{i\in [n]} |e^{x_i}-e^{y_i}|$.
Note that the max-plus and max-times semifields are isomorphic.

\subsection{Tropical cones and tropically convex sets: $(P,R)$-decomposition and homogenization}

We begin by recalling the definition of cones and by describing some relations between them and convex sets.

\begin{definition}
A set $\cV \subseteq \Rmax^n$ is called a
{\rm  (tropical) cone} if it is closed under (tropical)
addition and multiplication by scalars.  A cone $\cV$ in $\Rmax^n$ is said to be {\rm non-trivial} when $\cV\neq \{ \bzero \}$ and $\cV\neq \Rmax^n$.
\end{definition}

\begin{definition}
For $P,R\subseteq \Rmax^n$, we define the {\rm  (tropical) convex hull}
of $P$ to be:
\[
\conv (P):=\left\{\mysup_{y\in P} \lambda_y y\mid \lambda_y\in \Rmax \makebox{ for } y\in P\makebox{ and } \mysup_{y\in P}\lambda_y=\bunity \right\}
\]
and the {\rm (tropical) linear span} of $R$ or {\rm cone} generated by $R$ to be:
\[
\cone (R):=\left\{\mysup_{y\in R} \lambda_y y \mid  \lambda_y\in \Rmax \makebox{ for } y\in R \right\}\; ,
\]
where in both cases only a finite number of the
scalars $\lambda_y$ is not equal to $\bzero$. We will also consider the
{\rm (tropical) Minkowski sum} of
$\conv(P)$ and $\cone(R)$, which is
\[
\conv(P)\oplus\cone(R):=\left\{x\oplus y\mid x\in \conv(P),\; y\in \cone(R) \right\} \; .
\]
\end{definition}

Observe that $\cone(R)$ always contains the null vector $\bzero$, but $\conv(P)$ does not contain it in general.
For this reason, we always have $\conv(P)\subseteq\conv(P)\oplus\cone(R)$ and we do not
always have $\cone(R)\subseteq \conv(P)\oplus\cone(R)$.

\begin{definition}\label{def:PR}
Let $P,R\subseteq\Rmax^n$. If for a convex set
$\cC\subseteq\Rmax^n$ we have
\begin{equation}\label{e:GRbrief}
\cC=\conv(P)\oplus\cone(R)\; ,
\end{equation}
then~\eqref{e:GRbrief} is called a {\rm $(P,R)$-decomposition} of
$\cC$.
\end{definition}

For each convex set $\cC\subseteq\Rmax^n$ at least one decomposition
of the form~\eqref{e:GRbrief} exists: just take $P=\cC$ and
$R=\emptyset$. A canonical decomposition of the
form~\eqref{e:GRbrief} can be written for closed convex sets, by the
tropical analogue of Minkowski theorem, due to Gaubert and
Katz~\cite{GK-06,GK-07}.

\begin{definition}\label{def:hom}
For $C\subseteq\Rmax^n$, the set
\[
V_C=\{(\lambda x_1,\ldots,\lambda x_n, \lambda)\mid (x_1,\ldots,x_n)\in C, \lambda\in \Rmax\}\subset \Rmax^{n+1}
\]
is called the {\rm homogenization} of $C$.

For $x=(x_1,\ldots,x_n)\in\Rmax^n$, by abuse of notation, we shall
also denote the vector $(\lambda x_1,\ldots,\lambda x_n,
\lambda)\in\Rmax^{n+1}$ by $(\lambda x, \lambda)$, that is, we shall
use the identification of $\Rmax^{n+1}$ with $\Rmax^{n}\times \Rmax$
by the isomorphism $(z_1 , \dots, z_n , z_{n+1})\to ((z_1,
\dots ,z_n),z_{n+1})$. Thus we have $(\lambda x,\lambda)_i=\lambda
x_i$ for $i\in[n]$ and $(\lambda x,\lambda)_{n+1}=\lambda$.
\end{definition}

\begin{remark}
If $\cC\subseteq\Rmax^n$ is a convex set, then its homogenization $V_\cC\subseteq\Rmax^{n+1}$ is a cone.
A proof can be found in \cite[Lemma 2.12]{GK-07}.
\end{remark}

Reversing the homogenization means taking a section of a cone by a coordinate plane.
Below we take only sections of cones in $\Rmax^{n+1}$ by $x_{n+1}=\alpha$ (mostly with $\alpha=\bunity$),
and not by $x_i=\alpha$ with $i\in [n]$.

\begin{definition}\label{def:dehom}
For $\cV \subseteq\Rmax^{n+1}$ and $\alpha\in\Rmax$, the set
\begin{equation}
C^{\alpha}_{\cV}=\left\{x\in \Rmax^n\mid (x,\alpha)\in \cV\right\}
\end{equation}
is called a {\rm coordinate section} of $\cV$ by $x_{n+1}=\alpha$.
\end{definition}

Equivalently, the coordinate section of $\cV\subset \Rmax^{n+1}$ by
$x_{n+1}=\alpha$ is the image in
$\Rmax^{n}$ of $\cV\cap \{ x\in\Rmax^{n+1} \mid x_{n+1} = \alpha \}$ under the map
$(x_1,\dots,x_n,x_{n+1})\to (x_1,\dots,x_n)$.

The following property of coordinate section is standard (the proof is given for the reader's convenience).

\begin{proposition}\label{p:section}
Let $\cV \subseteq\Rmax^{n+1}$ be closed under
multiplication by scalars, and take any $\alpha\neq\bzero$. Then
$C_{\cV}^{\alpha}=\{\alpha x\mid x\in C_{\cV}^{\bunity}\}.$
\end{proposition}
\begin{proof}
If $x\in C_{\cV}^{\bunity}$ then $(x,\bunity)\in\cV$ and hence
$(\alpha x,\alpha)\in\cV$ and $\alpha x\in C_{\cV}^{\alpha}$. Thus
$\{\alpha x\mid x\in C_{\cV}^{\bunity}\}\subseteq C_{\cV}^{\alpha}$.
Similarly, $\{\alpha^{-1} x\mid x\in C_{\cV}^{\alpha}\}\subseteq
C_{\cV}^{\bunity}$. This implies $C_{\cV}^{\alpha}\subseteq \{\alpha x\mid
x\in C_{\cV}^{\bunity}\}$. (Indeed, if $x\in C_{\cV}^{\alpha}$ then
$\alpha^{-1}x\in C_{\cV}^{\bunity}$, and we have $x=\alpha y$ where
$y=\alpha^{-1}x\in C_{\cV}^{\bunity}$.)
\end{proof}

Let us write out a $(P,R)$-decomposition of a section of a cone
generated by a set $U\subseteq\Rmax^{n+1}$.

\begin{proposition}\label{p:gen-hom}
If $U\subseteq\Rmax^{n+1}$, $\cV=\cone(U)$
and the coordinate section $C^{\bunity}_{\cV}$ is non-empty,
then
\[
C^{\bunity}_{\cV}=\conv(P_U)\oplus\cone(R_U)
\]
where
\begin{equation}\label{gxrx-def}
P_U:=\{y\in \Rmax^n\mid \exists\mu \neq\bzero, (\mu y,\mu)\in U\}\; \makebox{ and }\;  R_U:=\{z\in \Rmax^n\mid (z,\bzero)\in U\} \; .
\end{equation}
\end{proposition}

\begin{proof}
Let us represent
\begin{align*}
U &
=\left\{ (u,u_{n+1})\in U\mid u\in \Rmax^n, u_{n+1} \neq \bzero\right\} \cup \left\{ (u,u_{n+1})\in U\mid u\in \Rmax^n, u_{n+1} = \bzero \right\} \\
& = \left\{ (\mu_y y,\mu_y)\in \Rmax^{n+1}\mid (\mu_y y,\mu_y)\in U,\; \mu_y \neq \bzero\right\} \cup \left\{(z,\bzero)\in\Rmax^{n+1}\mid (z,\bzero)\in U\right\}  \; .
\end{align*}
If $x\in C^{\bunity}_{\cV}$, i.e. $(x,\bunity) \in \cV =\cone (U)$, we have
\[
(x,\bunity) = \mysup_{(\mu_y y,\mu_y)\in U,\; \mu_y \neq \bzero} \lambda_y (\mu_y y,\mu_y)\oplus
\mysup_{(z,\bzero)\in U}\lambda_z (z,\bzero)
\]
for some $\lambda_y,\lambda_z\in\Rmax$, with only a finite number of
$\lambda_y,\lambda_z$ not equal to $\bzero$. Thus,
\[
\mysup_{(\mu_y y,\mu_y)\in U,\; \mu_y \neq \bzero}\lambda_y \mu_y =\bunity\;
\makebox{ and }\;x = \mysup_{(\mu_y y,\mu_y)\in U,\; \mu_y \neq \bzero}
\lambda_y \mu_y y \oplus \mysup_{(z,\bzero)\in U}\lambda_z z \; .
\]
It follows that  $x\in\conv(P_U)\oplus\cone(R_U)$.

Conversely, if $x\in\conv(P_U)\oplus\cone(R_U)$, we have
\[
x = \mysup_{y\in P_U} \lambda_y y \oplus \mysup_{z\in R_U}\lambda_z z \;
\]
for some $\lambda_y,\lambda_z\in\Rmax$, with $\mysup_{y\in P_U}\lambda_y =\bunity$ and only a finite number of
$\lambda_y,\lambda_z$ not equal to $\bzero$. Then,
\[
(x,\bunity) = \mysup_{y\in P_U} \lambda_y (y,\bunity) \oplus \mysup_{z\in R_U}\lambda_z (z,\bzero) \; .
\]
Since $(y,\bunity)\in \cV$ for $y\in P_U$ and $(z,\bzero)\in\cV$ for
$z\in R_U$, we conclude that $(x,\bunity)\in \cV$, and so $x\in
C^{\bunity}_{\cV}$.
\end{proof}

\begin{corollary}\label{CoroOfp:gen-hom}
Let $\cH=\conv (P) \oplus \cone (R)$, where $P,R\subset \Rmax^n$. Then,
if we define
\[
\cV:=\cone \left( \left\{ (x, \bunity) \mid x \in P\right\} \cup
\left\{ (y, \bzero) \mid y \in R\right\} \right) \; ,
\]
we have $C_{\cV}^{\bunity}=\cH$.
\end{corollary}

\begin{proof}
Let
\begin{equation}
\label{e:Udef}
U:=\left\{ (x, \bunity) \mid x \in P\right\} \cup \left\{ (y,
\bzero) \mid y \in R\right\}.
\end{equation}
Then, by Proposition~\ref{p:gen-hom}, we have
$C^{\bunity}_{\cV}=\conv(P_U)\oplus\cone(R_U)$, where $P_U$ and
$R_U$ are defined by~\eqref{gxrx-def}. With $U$ given
by~\eqref{e:Udef}, we have $P_U=P$ and $R_U=R$.

Indeed, let $y\in P_{U}.$ If $(\mu y,\mu )=(x,\bunity)$, with $\mu \neq \bzero$ and $x\in P$,
then $\mu =\bunity$ and $\mu y=x$, whence $y=x\in P$. On
the other hand, the relation $(\mu y,\mu )=(z,\bzero)$, with $\mu \neq \bzero$ and $z\in R$,
is impossible. Thus $P_{U}\subseteq P$. Conversely, if $y\in P$, then taking $\mu =\bunity,$ we have $(\mu
y,\mu )=(y,\bunity),$ so $y\in P_{U}.$ Thus $P\subseteq P_{U},$ which proves that $%
P_{U}=P.$

Now let $z\in R_{U}.$ Then $(z,\bzero)\in U$.
If $(z,\bzero)=(y,\bzero)$, with $y\in R$, then $z=y\in R$. On the other hand,
the relation $(z,\bzero)=(x,\bunity)$, with $x \in P$ is impossible.
Thus $R_{U}\subseteq R.$ Conversely, if $z\in R$, then $(z,\bzero)\in U$ by~\eqref{e:Udef}. Thus
$R\subseteq R_{U}$, which proves that $R_{U}=R$.

Hence $C_{\cV}^{\bunity}=\conv (P_U) \oplus \cone (R_U)=\conv (P) \oplus \cone (R)=\cH$.
\end{proof}

\subsection{Recessive elements}

We will use the following notions of recessive elements:

\begin{definition}\label{def:2.9}
Let $\cC\subseteq \Rmax^n$ be a convex set.
\begin{enumerate}[(i)]
\item Given $x\in \cC$, the set of {\rm recessive elements at $x$},
or {\rm locally recessive elements at $x$}, is defined as
\[
\rec_x \cC := \{z\in \Rmax^n\mid x\oplus \lambda z\in \cC \makebox{ for all } \lambda \in \Rmax\}\; .
\]
\item The set of {\rm globally recessive elements} of $\cC$, denoted by $\rec \cC$,
consists of the elements that are recessive at each element of $\cC$.
\end{enumerate}
\end{definition}

There is a close relation between recessive elements and
$(P,R)$-decompositions.

\begin{lemma}\label{lemma:new-lemma93}
If $\cC=\conv(P)\oplus\cone(R)$ as in~\eqref{e:GRbrief},
then $R\subseteq\rec \cC$.
\end{lemma}

\begin{proof}
Let $z\in R$. If $x\in \cC$, we have $x=p\oplus r$ for some $p\in \conv(P)$ and $r\in \cone(R)$. Then,
\[
x\oplus \lambda z=p\oplus (r\oplus \lambda z)\in \conv(P)\oplus \cone(R)=\cC \; ,
\]
for any $\lambda \in \Rmax$, because $r\oplus \lambda z \in \cone(R)$ as a consequence of fact that $\cone(R)$ is a cone. Since this holds for any $x\in \cC$ and $\lambda \in \Rmax$, we conclude that $z\in \rec \cC$.
\end{proof}

For closed convex sets, every locally recessive element is globally recessive:

\begin{proposition}[Gaubert and Katz~\cite{GK-07}]\label{p:GK}
If a convex set $\cC\subseteq \Rmax^n$ is closed,
then $\rec_x \cC\subseteq\rec \cC$ for all $x\in \cC$.
\end{proposition}

Proposition~\ref{p:GK} is proved in~\cite{GK-07} for the max-plus semifield, and
hence it follows also for the max-times semifield as these two semifields are isomorphic.

There are also other useful situations when a locally recessive element turns into
a globally recessive one.

\begin{lemma}\label{l:rec-supp}
Let $\cC\subseteq\Rmax^n$ be a convex set and $y\in \cC$.
If $z\in\rec_y \cC$ and
$\supp (y)\subseteq \supp (z)$, then $z\in\rec \cC$.
\end{lemma}

\begin{proof}
Since $z\in\rec_y \cC$ we have $y\oplus \lambda z\in \cC$ for all
$\lambda \in \Rmax$, and since $\supp (y)\subseteq \supp (z)$ we
have $y\oplus \lambda z=\lambda z$ for all $\lambda\geq \mu$, where
$\mu= \oplus_{i\in \supp(y)} y_i z_i^{-1}$ if $y\neq \bzero$ and
$\mu =\bzero$ otherwise. Given any $\beta \in \Rmax$, recalling that
$\Rmax$ denotes either $\Rmaxplus=(\R\cup \{-\infty\},\max,+)$ or
$\Rmaxtimes=([0,+\infty),\max,\times)$, we know that there exists
$\lambda \in \Rmax$ such that $\lambda > \beta \oplus \mu$. Then,
for any $x\in \cC$, we have $x\oplus \beta z=  x\oplus \beta
\lambda^{-1} \lambda z\in \cC$ because $\beta \lambda^{-1} \leq
\bunity$ and $x,\lambda z\in \cC$.
Thus, we conclude that $z\in\rec \cC$.
\end{proof}

Using the above observations, we now show that
$(P,R)$-decompositions can be combined, under certain
conditions.

\begin{theorem}\label{t:gathering}
Let $\{\cC_\ell \}$ be a family of convex sets in $\Rmax^n$,
each of which admits the following $(P,R)$-decomposition:
\[
\cC_\ell =\conv(P_\ell )\oplus \cone(R_\ell ) \; ,
\]
and let
$\cC:=\conv(\cup_{\ell} \cC_\ell )$. Then,
\begin{equation}\label{e:unit-gen}
\cC=\conv(\cup_{\ell} P_\ell )\oplus\cone(\cup_{\ell} R_\ell )
\end{equation}
if any of the following conditions hold:
\begin{enumerate}[(i)]
\item\label{e:unit-gen-C1} $R_\ell \subseteq\rec \cC$ for all $\ell$;
\item\label{e:unit-gen-C2} $\cC$ is closed;
\item\label{e:unit-gen-C3} For any $z\in R_\ell $
there exists $y\in\conv (P_\ell )$ such that $\supp(y)\subseteq\supp(z)$.
\end{enumerate}
\end{theorem}

\begin{proof}
We have:
\[
\cC_\ell=\conv(P_\ell )\oplus \cone(R_\ell )\subseteq \conv(\cup_{\ell}  P_\ell )\oplus\cone(\cup_{\ell} R_\ell ).
\]
As $\conv(\cup_{\ell}  P_\ell )\oplus\cone(\cup_{\ell}  R_\ell )$ is convex, it follows that
\[
\cC=\conv(\cup_{\ell}\cC_\ell)\subseteq\conv(\cup_{\ell}  P_\ell )\oplus\cone(\cup_{\ell}  R_\ell ).
\]

\eqref{e:unit-gen-C1} In this case $\cone(\cup_{\ell}
R_\ell)\subseteq\rec \cC$, and hence $\cC\oplus\cone(\cup_{\ell}
R_\ell )\subseteq \cC$.
We know that $P_\ell\subseteq\cC_\ell$, hence
$\conv(\cup_{\ell} P_\ell)\subseteq \conv(\cup_{\ell}\cC_\ell)=\cC$, whence
\[
\conv(\cup_{\ell} P_\ell )\oplus \cone(\cup_{\ell} R_\ell )\subseteq \cC\oplus\cone(\cup_{\ell} R_\ell )\subseteq \cC.
\]

Let us now prove that $R_\ell \subseteq\rec \cC$ holds for
cases~\eqref{e:unit-gen-C2} and~\eqref{e:unit-gen-C3}.

\eqref{e:unit-gen-C2} Each $z\in R_\ell $ is recessive at all $y\in P_\ell $, hence by Proposition~\ref{p:GK} it is globally recessive.

\eqref{e:unit-gen-C3} For $z\in R_{\ell}$, let
$y\in\conv(P_{\ell})$ be such that $\supp(y)\subseteq\supp(z)$. By
Lemma~\ref{lemma:new-lemma93} we have $z\in\rec\cC_{\ell}$, so in
particular $z\in\rec_y\cC_\ell$. It follows that $z \in \rec_y\cC$ because $\cC_\ell\subset \cC$. As $\supp(y)\subseteq\supp(z)$ and $z\in\rec_y\cC$,
we have $z\in\rec\cC$ by Lemma~\ref{l:rec-supp}.
\end{proof}

We will also need the following lemma.

\begin{lemma}\label{l:GatheringGenretorsCones}
Let $\{\cV_\ell\} =\{\cone(R_\ell )\} $ be a family of
cones
generated by the sets $R_\ell \subseteq \Rmax^n$ and let $\cV:=\cone(\cup_{\ell} \cV_\ell )$. Then, $\cV=\cone(\cup_{\ell} R_\ell )$.
\end{lemma}

\begin{proof}
We have $\cV_\ell=\cone(R_\ell )\subseteq \cone(\cup_{\ell} R_\ell )$ for all $\ell$, and so $\cup_{\ell}\cV_\ell\subseteq \cone(\cup_{\ell} R_\ell )$.
As $\cone(\cup_{\ell}  R_\ell )$ is a cone, it follows that
$\cV=\cone(\cup_{\ell}\cV_\ell)\subseteq \cone(\cup_{\ell}  R_\ell )$.

For the reverse inclusion, since $ R_\ell\subseteq \cV_\ell$ for all $\ell$,
we have $\cup_{\ell}  R_\ell \subseteq \cup_{\ell}\cV_\ell$, and so
$\cone(\cup_{\ell}  R_\ell )\subseteq \cone(\cup_{\ell}\cV_\ell)= \cV$.
\end{proof}

\section{Tropical semispaces}\label{s:semi}

In this section we aim to give a simpler proof
for the structure of semispaces in $\Rmax^n$,
originally described by Nitica and Singer~\cite{NS-07I,NS-07II},
and to introduce hemispaces in $\Rmax^n$
with some preliminary results on their relation with semispaces.

\subsection{Conical hemispaces, quasisectors and quasisemispaces}\label{ss:semicone}

We first introduce and study three objects called conical
hemispaces, quasisectors and quasisemispaces. The importance of
these lies in the fact that they will be the main tools for studying
hemispaces, sectors and semispaces (see
Definitions~\ref{def:convhemi},~\ref{def:cone23}
and~\ref{Def:Semispace} below) using the homogenization technique.

\begin{definition}\label{def:conhemi}
We call {\rm conical hemispace} a cone $\cV_1\subseteq\Rmax^n$, for
which there exists a cone $\cV_2\subseteq\Rmax^n$ such that
$\cV_1\cap\cV_2=\{\bzero\}$ and $\cV_1\cup\cV_2=\Rmax^n$. In this
case we call $(\cV_1,\cV_2)$ a {\rm joined pair of conical
hemispaces} (since $\cV_2$ is a conical hemispace as well). We say
that a joined pair $(\cV_1,\cV_2)$ of conical hemispaces is {\rm
non-trivial} when $\cV_1\neq \{ \bzero \}$ and $\cV_2\neq \{ \bzero
\}$.
\end{definition}

For completeness, we show the relationship between conical hemispaces and hemispaces (for the concept of hemispace, see the Introduction or
Definition~\ref{def:convhemi} below).

\begin{definition}
A subset $W\subseteq \Rmax^n$ is called {\rm wedge} if
$x\in W$ and $\lambda \in \Rmax$ imply $\lambda x\in W$.
\end{definition}

\begin{lemma}\label{lem:wedge}
If $W\subseteq \Rmax^n$ is a wedge, then $\complement W \cup \{\bzero\}\subseteq \Rmax^n$ is also a wedge.
\end{lemma}

\begin{proof}
Let $x\in \complement W \cup \{\bzero\}$ and $\lambda \in \Rmax$. We show that $\lambda x\in \complement W \cup \{\bzero\}$.
Indeed, if $\lambda x\in W$ and $\lambda x\not =\bzero,$ then $\lambda \not =\bzero$ and $x=\lambda^{-1}(\lambda x)\in W$, a contradiction.
\end{proof}

\begin{proposition}\label{p:conhemi}
A convex set $\cV\subseteq\Rmax^n$ is a conical
hemispace if and only if it is a hemispace and a cone.
\end{proposition}

\begin{proof}
If $\cV$ is a conical hemispace, then it is a cone and its
complement (not enlarged with $\bzero$) is a convex set. Thus $\cV$
is a hemispace and a cone, whence the ``only if'' part follows. Conversely,
assume (by contradiction) that $\cV$ is a hemispace and a cone, and
$\complement \cV\cup \{\bzero\}$ is a convex set and not a cone.
Then $\complement \cV \cup \{\bzero\}$ contains the sum of any two of its elements but it is not a cone,
so it is not a wedge. By Lemma~\ref{lem:wedge} $\cV$ is not a wedge, in
contradiction with the fact that it is a cone. Whence the ``if''
part follows.
\end{proof}

Note that conical hemispaces are almost the same as ``conical
halfspaces'' of Briec and Horvath~\cite{BH-08}. Indeed, the latter
``conical halfspaces'' are, in our terminology, hemispaces closed
under the multiplication by any non-null scalar.
In~\cite{BH-08} it is not required that $\bzero$ belongs to the "conical halfspace".

\begin{definition}\label{def:consec23}
For any $y\neq \bzero$ in $\Rmax^n$ and $i\in\supp(y)$, define the following sets:
\begin{equation}\label{conic-sec-compl}
\ConicSec_i(y):=\left\{x\in\Rmax^n\mid \mysup_{j\in \supp(y)} x_j y_j^{-1}\leq x_i y_i^{-1},
\;\makebox{and}\; x_j=\bzero\;\makebox{for all}\;j\notin\supp(y)\right\},
\end{equation}
which will be referred to as {\rm quasisectors} of type $i$.
\end{definition}

Since the complement of $\ConicSec_i(y)$ is
\begin{equation}\label{eq:Wiy}
\CompConicSec_i(y)=\left\{x\in\Rmax^n\mid \mysup_{j\in \supp(y)} x_j y_j^{-1}> x_i y_i^{-1},
\;\makebox{or}\; x_j>\bzero\;\makebox{for some}\;j\notin\supp(y) \right\},
\end{equation}
it follows that $\ConicSec_i(y)$ and $\CompConicSec_i(y)\cup
\{\bzero\}$ are both cones, so they form a joined pair of conical
hemispaces. Also note that $y\in \ConicSec_i(y)$ for all $i\in
\supp(y)$.

The following result appears in several places
(\cite{BSS,DS-04,Jos-05,GK-07,Ser-09-inLS}).

\begin{theorem}\label{t:multiorder}
Let $\cV\subseteq\Rmax^n$ be a cone and take $y\neq\bzero$
in $\Rmax^n$. Then $y\in \cV$ if and only if
\[
(\ConicSec_i(y)\setminus \{\bzero\})\cap\cV\not =\emptyset
\]
for each $i\in \supp(y)$.
\end{theorem}

\begin{proof} The ``only if'' part follows from the fact that $y\in \ConicSec_i(y)$ for $i\in \supp(y)$.

In order to prove the ``if'' part, assume that $i\in\supp(y)$ and
$x^i \in (\ConicSec_i(y)\setminus\{\bzero\})\cap \cV$. We claim that
$x^i_i\neq\bzero$. Indeed, if we had $x_i^i=\bzero$, then by
$x^i\in\ConicSec_i(y)$ and~\eqref{conic-sec-compl} we would have
$\oplus_{j\in\supp(y)}x_j^i y_j^{-1}=\bzero$ and $x^i_j=\bzero$ for all $j\notin\supp(y)$, hence $x^i=\bzero$, in contradiction with our assumption.
Furthermore, $y_i x_j^i\leq y_j x_i^i$ for all $j\in
[n]$. Then, $y$ can be written as a linear combination of
the $x^i$'s:
$$
y=\mysup_{i\in \supp(y)}\lambda_ix^i,
$$
where $\lambda_i=y_i(x_i^i)^{-1}$, therefore $y\in \cV$.
\end{proof}

Restating Theorem~\ref{t:multiorder} we get the following.

\begin{theorem}\label{t:multiorder2}
Let $\cV\subseteq\Rmax^n$ be a cone and take $y\neq\bzero$
in $\Rmax^n$. Then $y\notin \cV$ if and
only if $\cV\subseteq \CompConicSec_i(y)\cup \{\bzero\}$
for some $i\in \supp(y)$.
\end{theorem}

We are also interested in the following object.

\begin{definition}
A cone in $\Rmax^n$ is called a {\rm quasisemispace} at
$y\neq\bzero$ in $\Rmax^n$ if it is a maximal (with respect to
inclusion) cone not containing $y$.
\end{definition}

\begin{corollary}\label{c:semispaces}
There are exactly the cardinality of $\supp (y)$ quasisemispaces at
$y\neq\bzero$ in $\Rmax^n$. These are given by the cones $\CompConicSec_i(y)\cup \{\bzero\}$ for $i\in \supp (y)$.
\end{corollary}

\begin{proof}
Suppose that $\cV$ is a quasisemispace at $y$. Since it is a cone
not containing $y$, Theorem~\ref{t:multiorder2} implies that it is
contained in $\CompConicSec_i(y)\cup \{\bzero\}$ for some $i\in \supp (y)$.
By maximality, it follows that it coincides with
$\CompConicSec_i(y)\cup \{\bzero\}$.
\end{proof}

This statement shows that Theorem~\ref{t:multiorder2} is an instance
of a separation theorem in abstract convexity, since it says that
when $\cV$ is a cone in $\Rmax^n$, we have $y\notin \cV$ if and only
if there exists a quasisemispace
$\CompConicSec_i(y)\setminus\{\bzero\}$ (where $i\in\supp(y)$) in
$\Rmax^n$ that contains $\cV$ and does not contain $y$.
In particular, we obtain the following result.

\begin{corollary}\label{c:ext-repr}
Each non-trivial cone $\cV$ can be represented as the intersection
of the quasisemispaces $\CompConicSec_i(y)\cup \{\bzero\}$ containing it (where
$y\notin \cV$ and $i\in\supp(y)$), and for each complement $F$ of a
cone, the set $F\cup \{\bzero\}$ can be represented as the union of the
quasisectors $\ConicSec_i(y)$ contained in $F\cup \{\bzero\}$
(where $y\in F$ and $i\in\supp(y)$).
\end{corollary}

\begin{lemma}\label{l:ckiintersect}
Assume that $x,y\in\Rmax^n$ satisfy $\supp(x)\cap \supp(y)\neq \emptyset$.
Then, for any $i\in \supp(x)\cap \supp(y)$,
the non-null point $z$ with coordinates
\begin{equation}\label{zpoint}
z_j:=\min \left\{ x_i^{-1} x_j , y_i^{-1} y_j \right\}
\end{equation}
belongs to both $\ConicSec_i(x)$ and $\ConicSec_i(y)$.
\end{lemma}

\begin{proof}
Note that $z_j =\bzero$ for $j \not \in \supp(x)\cap \supp(y)$.
Moreover, since $z_i = \bunity$, we have $z_j x_j^{-1}\leq  x_i^{-1}
= z_i x_i^{-1}$ for all $j\in \supp(x)$. Then, we conclude that $z\in
\ConicSec_i(x)$. The proof of $z\in \ConicSec_i(y)$ is similar.
\end{proof}

Corollary~\ref{c:ext-repr} and Lemma~\ref{l:ckiintersect} imply the
following (preliminary) result on conical hemispaces.

\begin{theorem}\label{t:conhemsem}
For any joined pair $(\cV_1,\cV_2)$ of conical hemispaces in
$\Rmax^n$ there exist disjoint subsets $I,J$ of $[n]$ such that
\begin{equation}\label{eq:unions-123}
\begin{split}
\cV_1&= \cup \left\{\ConicSec_i(y)\mid \ConicSec_i(y)\subseteq \cV_1, y\in \cV_1, i\in I\right\} =\cone \left(\cup \left\{\ConicSec_i(y)\mid \ConicSec_i(y)\subseteq \cV_1, y\in \cV_1, i\in I\right\}\right) \; ,\\
\cV_2&=\cup \left\{\ConicSec_j(y)\mid \ConicSec_j(y)\subseteq \cV_2, y\in \cV_2, j\in J\right\} =\cone \left(\cup \left\{\ConicSec_j(y)\mid \ConicSec_j(y)\subseteq \cV_2, y\in \cV_2, j\in J\right\} \right) \; .
\end{split}
\end{equation}
\end{theorem}

\begin{proof}
 As $\cV_1\setminus \{ \bzero \}$ and $\cV_2\setminus \{ \bzero \}$ are complements of cones,
Corollary~\ref{c:ext-repr} yields that
\[
\begin{split}
\cV_1 & =\cup \left\{\ConicSec_i(y)\mid \ConicSec_i(y)\subseteq\cV_1, y\in \cV_1, i\in \supp(y)\right\} \; , \\
\cV_2 & =\cup \left\{ \ConicSec_j(y)\mid \ConicSec_j(y)\subseteq\cV_2, y\in \cV_2, j\in \supp(y)\right\} \; ,
\end{split}
\]
i.e. $\cV_1$ and $\cV_2$ are the union of the quasisectors contained
in them. We claim that the quasisectors contained in $\cV_1$ and
$\cV_2$ are of different type. To see this assume that, on the
contrary, there exist two points $y'\in \cV_1$, $y''\in \cV_2$ and
an index $i\in \supp(y')\cap \supp(y'')$ for which
$\ConicSec_i(y')\subseteq\cV_1$ and
$\ConicSec_i(y'')\subseteq\cV_2$. Then, by
Lemma~\ref{l:ckiintersect} applied to $y'$, $y''$ and $i$, we
conclude that the quasisectors $\ConicSec_i(y')$ and
$\ConicSec_i(y'')$, and so the conical hemispaces $\cV_1$ and
$\cV_2$, have a non-null point in common, which is a contradiction.

From the discussion above it follows that there exist disjoint subsets $I,J$ of $[n]$ such that
\[
\begin{split}
\cV_1 & =\cup \left\{\ConicSec_i(y)\mid \ConicSec_i(y)\subseteq \cV_1, y\in \cV_1, i\in I\right\} \; , \\
\cV_2 & =\cup \left\{ \ConicSec_j(y)\mid \ConicSec_j(y)\subseteq \cV_2, y\in \cV_2, j\in J\right\} \; .
\end{split}
\]
Finally, since the conical hemispaces $\cV_1$ and $\cV_2$ are cones,
the unions above coincide with their spans.
\end{proof}

\subsection{Tropical hemispaces, sectors and tropical semispaces}\label{ss:semiconv}

We now turn to convex sets using the homogenization technique.
Below we will be interested in the following objects.

\begin{definition}\label{def:convhemi}
We call {\rm  (tropical) hemispace} a convex set $\cH_1\subseteq\Rmax^n$,
for which there exists a convex set $\cH_2\subseteq\Rmax^n$ such
that $\cH_1\cap\cH_2=\emptyset$ and $\cH_1\cup\cH_2=\Rmax^n$. In
this case we call $(\cH_1,\cH_2)$ a {\rm complementary pair of hemispaces}.
We say that a complementary pair $(\cH_1,\cH_2)$ of hemispaces is {\rm non-trivial} when $\cH_1$ and $\cH_2$ are both non-empty.
\end{definition}

\begin{definition}\label{def:cone23}
For $y\in\Rmax^n$ and $i\in\supp(y)\cup\{n+1\}$, the coordinate
sections $\Sector_i(y) :=C^{\bunity}_{\ConicSec_i(y,\bunity)}$ are
called {\rm sectors of type} $i$.
\end{definition}

See Figure~\ref{f:segment-plane} below for an illustration of sectors in dimension $2$.

\begin{lemma}\label{LemmaCaractSectors}
For $y\in\Rmax^n$ and $i\in\supp(y)$, we have
\[
\begin{split}
&
\Sector_i(y) =
\left\{x\in \Rmax^n\mid \mysup_{j\in\supp(y)} x_j y_j^{-1}\oplus\bunity \leq x_i y_i^{-1}\;\makebox{and}\; x_j=\bzero\;\makebox{for all}\; j\notin\supp(y)\right\},\\
&\Sector_{n+1}(y)  =
\left\{x\in \Rmax^n\mid \mysup_{j\in\supp(y)} x_j y_j^{-1}\leq\bunity
\;\makebox{and}\; x_j=\bzero\;\makebox{for all}\; j\notin\supp(y)\right\},
\end{split}
\]
and so
\[
\begin{split}
&\CompSector_i(y)  =
\left\{x\in \Rmax^n\mid  \mysup_{j\in\supp(y)} x_j y_j^{-1}\oplus \bunity > x_i y_i^{-1}
\;\makebox{or}\; x_j>\bzero\;\makebox{for some}\; j\notin\supp(y)\right\}, \\
&\CompSector_{n+1}(y)  =
\left\{x\in \Rmax^n\mid \mysup_{j\in\supp(y)} x_j y_j^{-1}>\bunity
\;\makebox{or}\; x_j>\bzero\;\makebox{for some}\; j\notin\supp(y)\right\}.
\end{split}
\]
\end{lemma}

\begin{proof}
By Definition~\ref{def:consec23} applied to $(y,\bunity )$, we have
\[
\ConicSec_i(y,\bunity )= \left\{ (x,x_{n+1})\in\Rmax^{n+1}\mid \mysup_{j\in \supp(y)} x_j y_j^{-1}\oplus x_{n+1}\leq x_i y_i^{-1},  x_j=\bzero \makebox{ for all }j\notin \supp(y,\bunity) \right\}
\]
for $i\in\supp(y)$, and
\[
\ConicSec_{n+1}(y,\bunity )= \left\{(x,x_{n+1})\in\Rmax^{n+1}\mid \mysup_{j\in \supp(y)} x_j y_j^{-1} \oplus x_{n+1} \leq  x_{n+1}, x_j=\bzero \makebox{ for all } j\notin\supp(y,\bunity)\right\}
\]
because $\supp(y,\bunity)= \supp(y)\cup \{n+1\}$. Hence,
\begin{align*}
\Sector_i(y)=C^{\bunity}_{\ConicSec_i(y,\bunity)} = & \left\{ x\in\Rmax^{n}\mid (x,\bunity)\in\ConicSec_i(y,\bunity ) \right\} =\\
& \left\{ x\in\Rmax^{n}\mid \mysup_{j\in \supp(y)} x_j y_j^{-1}\oplus \bunity\leq x_i y_i^{-1},  x_j=\bzero \makebox{ for all }j\notin \supp(y) \right\}
\end{align*}
for $i\in\supp(y)$, and
\begin{align*}
\Sector_{n+1}(y)=C^{\bunity}_{\ConicSec_{n+1}(y,\bunity)}=&\left\{ x\in\Rmax^{n}\mid (x,\bunity)\in\ConicSec_{n+1}(y,\bunity ) \right\} =\\
&\left\{x\in\Rmax^{n+1}\mid \mysup_{j\in \supp(y)} x_j y_j^{-1} \leq \bunity , x_j=\bzero \makebox{ for all } j\notin\supp(y)\right\}
\end{align*}
since
$\mysup_{j\in \supp(y)} x_j y_j^{-1}\oplus \bunity  \leq  \bunity $ is equivalent to $\mysup_{j\in \supp(y)} x_j y_j^{-1} \leq  \bunity $.
\end{proof}

Let us make the following observation which will be useful in the next section.

\begin{lemma}\label{l:sn+1incr}
Let $y\in\Rmax^n$ and
$\alpha,\beta\in\Rmax$ be such that $\alpha\leq\beta$. Then $\Sector_{n+1}(\alpha
y)\subseteq\Sector_{n+1}(\beta y)$.
\end{lemma}

\begin{proof}
If $\alpha=\bzero$ then $\Sector_{n+1}(\alpha y)=\{ \bzero \}$, and
the inclusion follows since $\bzero\in\Sector_{n+1}(z)$ for each
$z\in\Rmax^n$. If $\bzero<\alpha\leq\beta$ then by Lemma~\ref{LemmaCaractSectors} we have
\[
\begin{split}
\Sector_{n+1}(\alpha y)  &= \left\{x\in \Rmax^n\mid \mysup_{j\in\supp(y)} x_j y_j^{-1}\leq\alpha
\;\makebox{and}\; x_j=\bzero\;\makebox{for all}\; j\notin\supp(y)\right\},\\
\Sector_{n+1}(\beta y)  &= \left\{x\in \Rmax^n\mid \mysup_{j\in\supp(y)} x_j y_j^{-1}\leq\beta
\;\makebox{and}\; x_j=\bzero\;\makebox{for all}\; j\notin\supp(y)\right\}
\end{split}
\]
and the inclusion follows since each $x\in\Sector_{n+1}(\alpha
y)$ satisfies $\oplus_{i\in\supp(y)}
x_iy_i^{-1}\leq\alpha\leq\beta$ and $x_j=\bzero$ for all $j\notin\supp(y)$.
\end{proof}

\begin{remark} Both $\Sector_i(y)$ and $\CompSector_i(y)$
(for $i\in\supp(y)\cup\{n+1\}$)
are convex sets and complements of each other,
hence they form a complementary pair of hemispaces.
\end{remark}

\begin{remark}
The notation for sectors and semispaces
is reversed as compared to the notation in
Nitica and Singer~\cite{NS-07I,NS-07II,NS-07}.
\end{remark}

\begin{theorem}\label{t:multiorder1-conv}
Let $y\in\Rmax^n$ and let $\cC\subseteq\Rmax^n$ be convex.
Then $y\in \cC$ if and only if
\begin{equation}\label{new-eq-56}
\Sector_i(y)\cap\cC\not =\emptyset
\end{equation}
for each $i\in\supp(y)$ and for $i=n+1$.
\end{theorem}

\begin{proof}
The ``only if'' part is trivial, since $\Sector_i(y)$ contains $y$
for each $i\in\supp(y)$ and for $i=n+1$.

For the ``if'' part, consider the homogenization $V_\cC$ of $\cC$.
If \eqref{new-eq-56} is satisfied, then for each $i\in\supp(y)$ and for $i=n+1$ there exist
$x^i\in\Sector_i(y)\cap \cC=C^{\bunity}_{\ConicSec_i(y,\bunity)}\cap \cC$, which implies $(x^i,\bunity)\in
(\ConicSec_i(y,\bunity)\setminus \{(\bzero,\bzero)\})\cap V_\cC$. By
Theorem~\ref{t:multiorder}, it follows that $(y,\bunity)\in V_\cC$,
and so $y\in \cC$.
\end{proof}

Restating Theorem~\ref{t:multiorder1-conv} we obtain the following.

\begin{theorem}\label{t:multiorder2-conv}
Let $\cC\subseteq\Rmax^n$ be a convex set and take
$y\in\Rmax^n$. Then $y\notin \cC$ if and only if $\cC\subseteq
\CompSector_i(y)$ for some $i\in\supp(y)$ or $i=n+1$.
\end{theorem}

\begin{definition}\label{Def:Semispace}
A convex set of $\Rmax^n$ is called a {\rm (tropical) semispace} at
$y\in\Rmax^n$ if it is a maximal (with respect to inclusion)
convex set of $\Rmax^n$ not containing $y$.
\end{definition}

\begin{corollary}\label{c:semispaces-conv}
There are exactly the cardinality of $\supp(y)$ plus one semispaces at $y\in \Rmax^n$.
These are given by the convex sets
$\CompSector_i(y)$ for $i\in \supp(y)$ and $i=n+1$.
\end{corollary}

\begin{proof}
Suppose that $\cC$ is a semispace at $y\in \Rmax^n$. Since it is a convex
set not containing $y$, Theorem~\ref{t:multiorder2-conv} implies
that it is contained in $\CompSector_i(y)$ for some $i\in \supp (y)$
or $i=n+1$. By maximality, it follows that it coincides with
$\CompSector_i(y)$.
\end{proof}

The following corollary corresponds to Corollary~\ref{c:ext-repr}.

\begin{corollary}\label{c:ext-repr-conv}
Each convex set $\cC\subseteq\Rmax^n$ can be represented as the
intersection of the semispaces $\CompSector_i(y)$ containing it (where
$y\notin \cC$ and $i\in\supp(y)$ or $i=n+1$), and each complement $F$ of a
convex set can be represented as the union of the sectors
$\Sector_i(y)$ contained in $F$
(where $y\in F$ and $i\in\supp(y)$ or $i=n+1$).
\end{corollary}

\begin{lemma}\label{l:csiintersect}
For any two points $x,y\in\Rmax^n$ and $i\in\supp(x)\cap\supp(y)$ or $i=n+1$
the intersection $\Sector_i(x)\cap \Sector_i(y)$ is non-empty.
\end{lemma}

\begin{proof}
Consider the points $(x,\bunity)$ and $(y,\bunity)$ and observe
that for $\cV:=\ConicSec_i(x,\bunity)\cap \ConicSec_i(y,\bunity)$ we
have:
\[
\begin{split}
C_{\cV}^{\bunity} & =\left\{z\in\Rmax^n \mid (z,\bunity)\in \cV\right\} =\left\{z\in\Rmax^n \mid (z,\bunity)\in (\ConicSec_i(x,\bunity)\cap \ConicSec_i(y,\bunity) )\right\} \\
& =\left\{z\in\Rmax^n \mid (z,\bunity)\in \ConicSec_i(x,\bunity)\right\}\cap
\left\{z\in\Rmax^n \mid (z,\bunity)\in  \ConicSec_i(y,\bunity) \right\} \\
& = C_{\ConicSec_i(x,\bunity)}^{\bunity} \cap C_{\ConicSec_i(y,\bunity)}^{\bunity}=
\Sector_i(x)\cap \Sector_i(y)\; .
\end{split}
\]

For any $i\in \supp(x,\bunity)\cap\supp(y,\bunity)=
(\supp(x)\cap\supp(y))\cup \{n+1\}$, Lemma~\ref{l:ckiintersect}
applied to $(x,\bunity)$ and $(y,\bunity)$ provides a non-null point
$z$ in $\cV=\ConicSec_i(x,\bunity)\cap \ConicSec_i(y,\bunity)$. This
point is defined by~\eqref{zpoint} applied to $(x,\bunity)$ and
$(y,\bunity)$, so $z_{n+1}=\min \{ x_i^{-1} , y_i^{-1} \}$ if $i\in
\supp(x)\cap\supp(y)$ and $z_{n+1}=\bunity$ if $i=n+1$. In both
cases we have $z_{n+1}\neq \bzero$, and then we conclude that
$z_{n+1}^{-1}(z_1,\ldots ,z_n)\in C_{\cV}^{\bunity}=
\Sector_i(x)\cap \Sector_i(y)$ because $(z_{n+1}^{-1}z_1,\ldots
,z_{n+1}^{-1}z_n,\bunity)=z_{n+1}^{-1}z\in \cV$.
\end{proof}

Corollary~\ref{c:ext-repr-conv} and Lemma~\ref{l:csiintersect} imply
the following (preliminary) result on general hemispaces (an
analogue of Theorem~\ref{t:conhemsem}).

\begin{theorem}\label{t:convhemsem}
For any complementary pair of hemispaces $\cH_1$ and $\cH_2$
there exist disjoint subsets $I,J\subseteq [n+1]$ such that
\begin{equation}\label{union-hemi-12}
\begin{split}
\cH_1&=\cup\left\{\Sector_i(y)\mid \Sector_i(y)\subseteq\cH_1, i\in I, y\in \cH_1\right\} =\conv \left(\cup \left\{\Sector_i(y)\mid \Sector_i(y)\subseteq\cH_1, i\in I, y\in \cH_1\right\} \right) ,\\
\cH_2&=\cup\left\{\Sector_j(y)\mid \Sector_j(y)\subseteq\cH_2, j\in J, y\in\cH_2\right\} =\conv\left(\cup \left\{\Sector_j(y)\mid \Sector_j(y)\subseteq\cH_2, j\in J, y\in \cH_2\right\}\right) .
\end{split}
\end{equation}
\end{theorem}

\begin{proof}
As $\cH_1$ and $\cH_2$ are complements of convex sets,
Corollary~\ref{c:ext-repr-conv} yields that $\cH_1$ and $\cH_2$ are
the union of the sectors contained in them. The sectors contained in
$\cH_1$ and $\cH_2$ should be of different type, since otherwise
there exist two points  $y'\in \cH_1, y''\in \cH_2,$ and an index
$i\in (\supp(y')\cap \supp(y''))\cup \{n+1\}$ for which
$\Sector_i(y')\subseteq\cH_1$ and $\Sector_i(y'')\subseteq\cH_2$.
Then, by Lemma~\ref{l:csiintersect} applied to $y'$, $y''$ and $i$,
we conclude that the hemispaces $\cH_1$ and $\cH_2$ have a common
point, which is a contradiction.

Finally, since the hemispaces $\cH_1$ and $\cH_2$ are convex sets,
the unions in~\eqref{union-hemi-12} coincide with their convex hulls.
\end{proof}

Theorem~\ref{t:convhemsem} can be used to describe hemispaces in
the case $n=2$. Indeed, in this case, the non-empty and disjoint
sets $I$ and $J$ appearing in its formulation should satisfy $I\cup
J\subseteq \{1,2,3\}$. It follows that one of the sets $I$ or $J$
consists of only one index. Thus, one of the hemispaces $\cH_1$ or
$\cH_2$ is the union of sectors of the same type. By careful
inspection of all possible cases, for this hemispace we obtain the
sets shown on the diagrams of Figures~\ref{f:hemispace-plane}
and~\ref{f:hemispace-line}. Using the form of typical
(tropical) segments
on the plane, shown on the left-hand side of
Figure~\ref{f:segment-plane}, it can be checked graphically that all
these sets and their complements are indeed convex sets (and
hence, indeed, hemispaces).
All figures are done in the max-times semifield $\Rmaxtimes$.

\begin{figure}
\begin{center}
\begin{picture}(0,0)%
\includegraphics{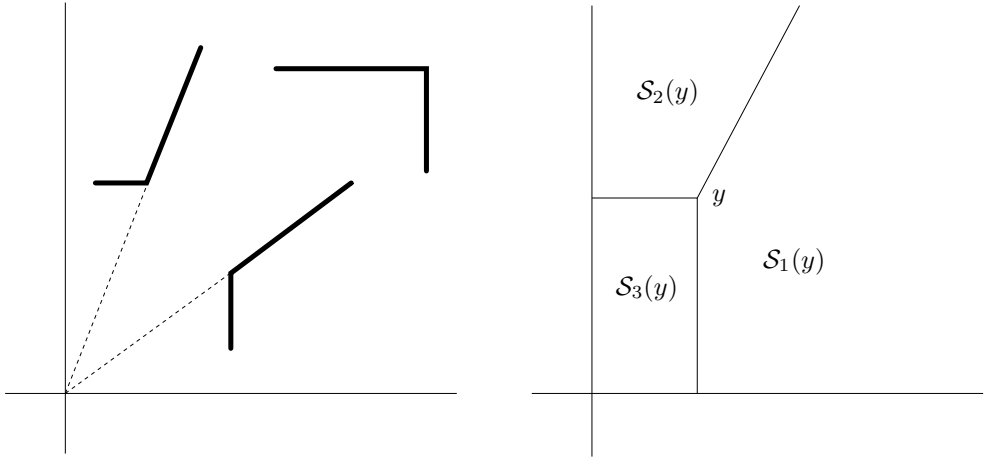}%
\end{picture}%
\setlength{\unitlength}{1657sp}%
\begingroup\makeatletter\ifx\SetFigFont\undefined%
\gdef\SetFigFont#1#2#3#4#5{%
  \reset@font\fontsize{#1}{#2pt}%
  \fontfamily{#3}\fontseries{#4}\fontshape{#5}%
  \selectfont}%
\fi\endgroup%
\begin{picture}(14669,6839)(879,-6878)
\put(11476,-3076){\makebox(0,0)[lb]{\smash{{\SetFigFont{10}{12.0}{\rmdefault}{\mddefault}{\updefault}{\color[rgb]{0,0,0}$y$}%
}}}}
\put(12241,-4066){\makebox(0,0)[lb]{\smash{{\SetFigFont{10}{12.0}{\rmdefault}{\mddefault}{\updefault}{\color[rgb]{0,0,0}$\cS_1(y)$}%
}}}}
\put(10036,-4471){\makebox(0,0)[lb]{\smash{{\SetFigFont{10}{12.0}{\rmdefault}{\mddefault}{\updefault}{\color[rgb]{0,0,0}$\cS_3(y)$}%
}}}}
\put(10351,-1546){\makebox(0,0)[lb]{\smash{{\SetFigFont{10}{12.0}{\rmdefault}{\mddefault}{\updefault}{\color[rgb]{0,0,0}$\cS_2(y)$}%
}}}}
\end{picture}%
\end{center}
\caption{Max-times segments (on the left) and sectors based at a point $y$ with
full support $\{1,2\}$ (on the right) in dimension $2$.}\label{f:segment-plane}
\end{figure}

\begin{figure}
\begin{center}
\includegraphics[width=9.5cm]{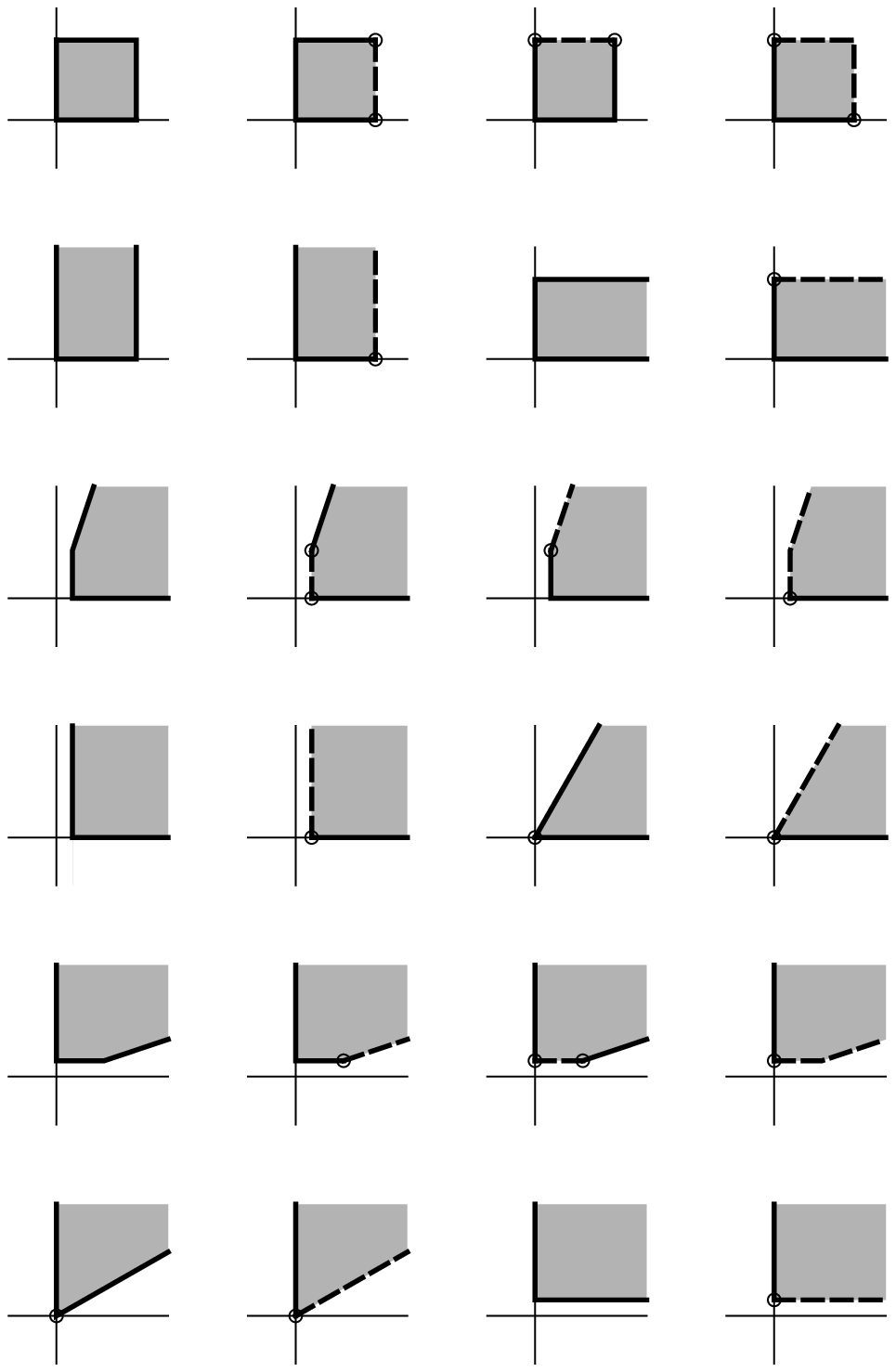}
\end{center}
\caption{The hemispaces in dimension $2$ which can be obtained as
unions of sectors of the same type based at points with full support $\{1,2\}$.}\label{f:hemispace-plane}
\end{figure}

\begin{figure}
\begin{center}
\includegraphics[width=9.5cm]{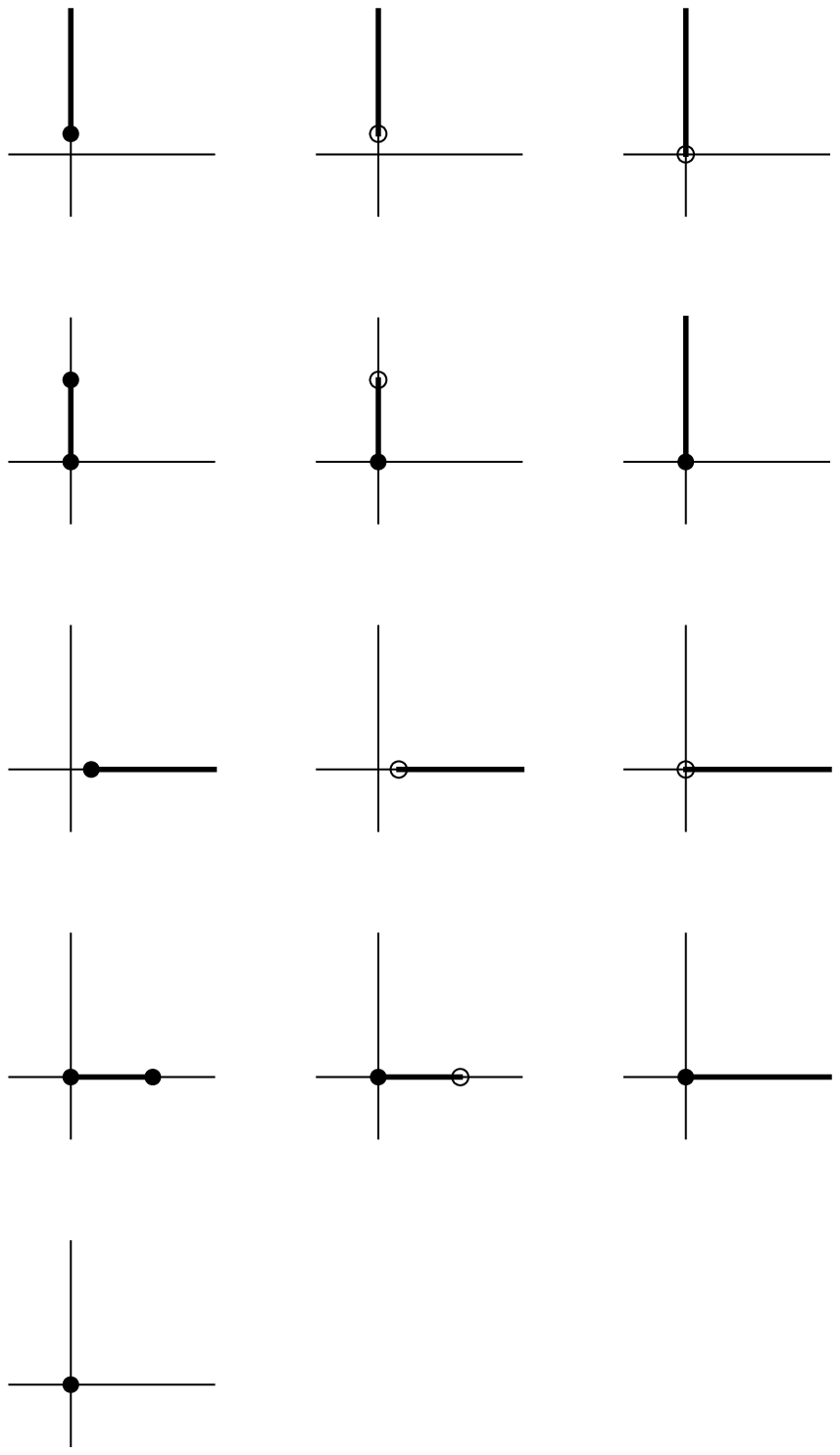}
\end{center}
\caption{The hemispaces in dimension $2$ which can be obtained as unions of sectors of
the same type based at points with non-full support ($\{1\}$ or $\{2\}$).}\label{f:hemispace-line}
\end{figure}

\section{Tropical hemispaces}\label{s:hemi}

\subsection{Homogenization and $(P,R)$-decompositions.}\label{ss:hemihomog}

Let us start with $(P,R)$-decompositions of quasisectors and
sectors.

\begin{proposition}\label{p:ckigenconv}
For $y\in \Rmax^n$ and $i\in\supp(y)$, the quasisectors
$\ConicSec_i(y)$ and the sectors $\Sector_i(y)$ and
$\Sector_{n+1}(y)$ can be represented as
\begin{equation}\label{e:compl-int-repr-conv}
\begin{split}
\ConicSec_i(y)& = \cone \left(\left\{ \uvector{i}\oplus y_j y_i^{-1}\uvector{j}\mid j\in \supp(y) \right\} \right) \; ,\\
\Sector_i(y)& = \left\{ y_i\uvector{i} \right\} \oplus \cone \left( \left\{ \uvector{i}\oplus y_jy_i^{-1}\uvector{j}\mid j\in \supp(y)\right\} \right) \; ,\\
 \Sector_{n+1}(y)& = \conv \left( \left\{\bzero\right\} \cup \left\{ y_j\uvector{j}\mid j\in\supp(y)\right\} \right)\; .
\end{split}
\end{equation}
\end{proposition}

\begin{proof}
We claim that if $x\in \ConicSec_i(y)$, then
\[
x=\bigoplus_{j\in \supp(y)}  y_i y_j^{-1}x_j
(\uvector{i}\oplus y_j y_i^{-1}\uvector{j}) \; .
\]
Indeed, we have
\[
\left(\bigoplus_{j\in \supp(y)}  y_i y_j^{-1}x_j
(\uvector{i}\oplus y_j y_i^{-1}\uvector{j})\right)_i=\bigoplus_{j\in \supp(y)}  y_i y_j^{-1}x_j=x_i
\]
since $i\in \supp(y)$ and $ y_i y_j^{-1}x_j \leq x_i$ for all $j \in \supp
(y)$ by Definition~\ref{def:consec23}. Furthermore for $k\in\supp(y)\setminus\{i\}$ we have
\[
\left(\bigoplus_{j\in \supp(y)}  y_i y_j^{-1}x_j
(\uvector{i}\oplus y_j y_i^{-1}\uvector{j})\right)_k= y_i y_k^{-1}x_k y_k y_i^{-1}=x_k,
\]
and for $k\in [n] \setminus \supp(y)$ we have
\[
\left(\bigoplus_{j\in \supp(y)}  y_i y_j^{-1}x_j
(\uvector{i}\oplus y_j y_i^{-1}\uvector{j})\right)_k= \bigoplus_{j\in \supp(y)}  y_i y_j^{-1}x_j
(\uvector{i}\oplus y_j y_i^{-1}\uvector{j})_k=\bzero=x_k.
\]
This proves our claim.
Using this property,
we conclude that
\[
\ConicSec_i(y) \subseteq \cone \left( \left\{
\uvector{i}\oplus y_j y_i^{-1}\uvector{j}\mid j\in \supp(y) \right\}\right)\; .
\]
For the converse inclusion, let us show that the vector
$\uvector{i}\oplus y_j y_i^{-1}\uvector{j}$ belongs to
$\ConicSec_i(y)$ for any $j \in \supp (y)$. Indeed, we have $\left(
\uvector{i}\oplus y_j y_i^{-1}\uvector{j}\right)_k=\bzero$ for any
$k\in [n]\setminus\{i,j\}$, and so in particular for any $k\in
[n]\setminus \supp(y)$, and
\[
\bigoplus_{k\in\supp(\uvector{i}\oplus y_j y_i^{-1}\uvector{j})}
\left( \uvector{i}\oplus y_j y_i^{-1}\uvector{j}\right)_k y_k^{-1}=
y_i^{-1}\oplus y_j y_i^{-1} y_j^{-1}= y_i^{-1}=\left(
\uvector{i}\oplus y_j y_i^{-1}\uvector{j}\right)_iy_i^{-1} \; .
\]
Thus, $\uvector{i}\oplus y_j y_i^{-1}\uvector{j}\in\ConicSec_i(y)$
by Definition~\ref{def:consec23}. Since $\ConicSec_i(y)$ is a
cone and $\uvector{i}\oplus y_j y_i^{-1}\uvector{j}\in
\ConicSec_i(y)$ for any $j\in\supp(y)$, we conclude that
\[
\cone \left( \left\{ \uvector{i}\oplus y_j y_i^{-1}\uvector{j}\mid j\in
\supp(y) \right\} \right) \subseteq \ConicSec_i(y)\; .
\]
This completes the proof of the first equality in~\eqref{e:compl-int-repr-conv}.

From the first equality in~\eqref{e:compl-int-repr-conv} it follows that, given $y\in \Rmax^{n}$, for all
$i\in \supp(y,\bunity ) = \supp(y)\cup \{n+1\}$ we have
$\ConicSec_{i}(y,\bunity)=\cone(U_{i})$, where
\begin{equation}\label{eq:label1-ivan}
U_{i}=\left\{ \uvector{i}\oplus (y,\bunity)_{j}(y,\bunity)_{i}^{-1}\uvector{j}\mid j\in \supp(y)\cup
\{n+1\}\right\} \; .
\end{equation}
Hence by Definition~\ref{def:cone23} and Proposition~\ref{p:gen-hom}, it follows
that for all $i\in \supp(y)\cup \{n+1\}$,
\begin{equation}\label{eq:label2-ivan}
\Sector_{i}(y)=C_{\ConicSec_{i}(y,\bunity)}^{\bunity}=
C_{\cone(U_{i})}^{\bunity}=\conv (P_{U_{i}})\oplus \cone(R_{U_{i}}) \; ,
\end{equation}
where
\begin{equation}\label{eq:label3-ivan}
P_{U_{i}}=\left\{ y\in \Rmax^{n}\mid \exists \mu \neq \bzero,(\mu y,\mu )\in U_{i}\right\} \; ,
\end{equation}
and
\begin{equation}\label{eq:label4-ivan}
R_{U_{i}}=\left\{ z\in \Rmax^{n}\mid (z,\bzero)\in U_{i} \right\} \; .
\end{equation}

Let $i\in \supp(y)$ (hence $i\leq n$). Then by~\eqref{eq:label1-ivan} we have
\[
U_{i}=\left\{\uvector{i}\oplus y_{j}y_{i}^{-1}\uvector{j}\mid j\in \supp(y)\right\} \cup
\left\{ \uvector{i}\oplus y_{i}^{-1}\uvector{n+1}\right\} \; .
\]
Therefore by~\eqref{eq:label3-ivan} $z\in P_{U_{i}}$ if and only if
there exists $\mu \neq \bzero$ such that $(\mu z,\mu
)=\uvector{i}\oplus y_{i}^{-1}\uvector{n+1}$ which yields $\mu
=y_{i}^{-1}$ and $y_{i}^{-1}z=\mu z=\uvector{i}$, whence
$z=y_{i}\uvector{i}$. Thus
$P_{U_{i}}=\left\{y_{i}\uvector{i}\right\}$. Furthermore,
by~\eqref{eq:label4-ivan} $z\in R_{U_{i}}$ if and only if
$(z,\bzero)=\uvector{i}\oplus y_{j}y_{i}^{-1}\uvector{j}$ for some
$j\in \supp(y)$. Consequently, by~\eqref{eq:label2-ivan}, we obtain
the second equality of~\eqref{e:compl-int-repr-conv}.

Finally, let $i=n+1.$ Then by~\eqref{eq:label1-ivan},
\[
U_{n+1}=\left\{ \uvector{n+1}\oplus y_{j}\uvector{j}\mid j\in \supp(y)\right\}\cup \left\{\uvector{n+1} \right\},
\]
whence $(z,\bzero)\notin U_{n+1}$ for all $z\in \Rmax^{n},$ and
hence by~\eqref{eq:label4-ivan}, $R_{U_{n+1}}=\emptyset $.
Furthermore, for $\mu \neq \bzero$ we have $(\mu z,\mu )\in U_{n+1}$
if and only if either $(\mu z,\mu )=\uvector{n+1}\oplus
y_{j}\uvector{j}$ for some $j\in \supp(y)$ or $(\mu z,\mu)=\uvector{n+1}$.
In the first case we obtain $\mu =\bunity$ and
$z=\mu z=y_{j}\uvector{j}$ for some $j\in \supp(y)$, and in the
second case we obtain $\mu =\bunity$ and $z=\mu z=\bzero$. Thus
by~\eqref{eq:label3-ivan}, $P_{U_{n+1}}=\left\{\bzero\right\} \cup
\left\{ y_j\uvector{j}\mid j\in\supp(y)\right\}$, whence by
$R_{U_{n+1}}=\emptyset $ and~\eqref{eq:label2-ivan},
we obtain the third equality of~\eqref{e:compl-int-repr-conv}.
\end{proof}

We now obtain $(P,R)$-decompositions of hemispaces (respectively,
conical hemispaces) by uniting the $(P,R)$-decompositions of sectors
(respectively, quasisectors) contained in them.

\begin{theorem}\label{t:onetwo}
For any hemispace $\cH\subset \Rmax^n$ (resp. any conical hemispace
$\cV\subset \Rmax^n$) a $(P,R)$-decomposition can be obtained by
uniting the $(P,R)$-decompositions given in~\eqref{e:compl-int-repr-conv} of all
$\Sector_i(y)\subseteq\cH$ (resp. $\ConicSec_i(y)\subseteq\cV$),
where $y\in\cH$ and $i\in\supp(y)\cup\{n+1\}$ (resp. $y\in\cV$ and
$i\in\supp(y)$).

If $\cH$ is a hemispace, the resulting $(P,R)$-decomposition is given by
\begin{equation}\label{e:PRhemi}
\begin{split}
P & =
\begin{cases}
\{y_i\uvector{i}\mid \Sector_i(y)\subseteq\cH\},\ \text{ if there is
no $y\in \Rmax^n$ such that
$\Sector_{n+1}(y)\subseteq\cH$} \\
\{\bzero\}\cup \{y_i\uvector{i}\mid \Sector_i(y)\subseteq\cH\}\cup
\{y_j\uvector{j}\mid\Sector_{n+1}(y)\subseteq\cH,\; j\in\supp(y)\}\text{ otherwise}.
\end{cases}\\
R & =\left\{\uvector{i}\oplus y_j y_i^{-1} \uvector{j}\mid\Sector_i(y)\subseteq\cH,\; j\in\supp(y)\right\} \; ,
\end{split}
\end{equation}
and if $\cV$ is a conical hemispace, then we have
\begin{equation}\label{e:PRquasihemi}
P=\emptyset \; ,\
R=\left\{\uvector{i}\oplus y_j y_i^{-1} \uvector{j}\mid\ConicSec_i(y)\subseteq\cV,\; j\in\supp(y)\right\} \; .
\end{equation}
\end{theorem}

\begin{proof}
By Theorem~\ref{t:convhemsem} any hemispace $\cH$ can be represented
as the convex hull of all the sectors contained in $\cH$. Consider
the $(P,R)$-decomposition of sectors given in the last two lines
of~\eqref{e:compl-int-repr-conv}. The pair of sets $(P,R)$ which
determines the $(P,R)$-decomposition of the sector $\Sector_i(y)$,
for any $y\in\Rmax^n$ and $i\in\supp(y)$, satisfy
Condition~\eqref{e:unit-gen-C3} of Theorem~\ref{t:gathering} due to
the fact that $\supp(y_i\uvector{i})\subseteq
\supp(\uvector{i}\oplus y_j y_i^{-1}\uvector{j})$ for all
$j\in\supp(y)$, and the pair of sets deterining the
$(P,R)$-decomposition of the sector $\Sector_{n+1}(y)$ satisfies
this condition trivially (since $R$ is empty). Therefore we can
combine all the $(P,R)$-decompositions of the sectors contained in $\cH$
(in other words, take the unions of all $P$ and all $R$ separately) to obtain a
$(P,R)$-decomposition of $\cH$. To form the set $P$, let us first
collect, using Theorem~\ref{t:convhemsem} and the second line
of~\eqref{e:compl-int-repr-conv}, all the vectors $y_i e^i$ such
that $\Sector_i(y)\subseteq\cH$ (where $i\in\supp(y)$). If we have
$\Sector_{n+1}(y)\subseteq\cH$ for some $y\in\Rmax^n$ then, using
Theorem~\ref{t:convhemsem} and the third line
of~\eqref{e:compl-int-repr-conv}, we also add the zero vector and
all the vectors $y_j\uvector{j}$, where $j\in\supp(y)$. This
explains the expression for $P$ in~\eqref{e:PRhemi}, in both cases.
The set $R$ is composed of the vectors $\uvector{i}\oplus
y_jy_i^{-1}\uvector{j}$ appearing on the second line
of~\eqref{e:compl-int-repr-conv}, such that
$\Sector_i(y)\subseteq\cH$ and $j\in\supp(y)$. This explains the
last line of~\eqref{e:PRhemi}.

By Theorem~\ref{t:conhemsem}, any conical hemispace $\cV$ is the
linear span of all the quasisectors contained in $\cV$.
Consider the $(P,R)$-decomposition of quasisectors given in the
first line of~\eqref{e:compl-int-repr-conv}.
By Lemma~\ref{l:GatheringGenretorsCones} the union of all the sets $R$ appearing in these
$(P,R)$-decompositions of the quasisectors contained in $\cV$ gives the set $R$ appearing in a
$(P,R)$-decomposition of $\cV$ (in which $P=\emptyset$).
By Theorem~\ref{t:conhemsem} and the
first line of~\eqref{e:compl-int-repr-conv}, $R$ consists of all the
vectors $\uvector{i}\oplus y_j y_i^{-1}\uvector{j}$ such that
$\ConicSec_i(y)\subseteq\cV$ and $j\in\supp(y)$. This
shows~\eqref{e:PRquasihemi}.
\end{proof}

Let us make an observation on the $(P,R)$-decomposition of
Theorem~\ref{t:onetwo}.

\begin{lemma}\label{l:sizspanr}
Let $\cH\subseteq \Rmax^n$ be a hemispace, $z\in\Rmax^n$, and let $R$ be defined
by the last line of~\eqref{e:PRhemi}. If
$\Sector_i(z)\subseteq\cH$ then $z\in\cone(R)$.
\end{lemma}

\begin{proof}
Since $\Sector_i(z)\subseteq\cH$, by the last line of~\eqref{e:PRhemi} the set $R$
contains all the vectors of the form
$\uvector{i}\oplus z_jz_i^{-1}\uvector{j}$ for $j\in\supp(z)$.
Representing
\[
z=z_i \left(\bigoplus_{j\in\supp(z)} (\uvector{i}\oplus z_jz_i^{-1}\uvector{j})\right)\, ,
\]
we conclude that $z\in\cone(R)$.
\end{proof}

We shall need the following characterization of joined pairs of
conical hemispaces by means of sections.

\begin{lemma}\label{l:CharaHemiWithSections}
Let $\cV_1,\cV_2\subseteq \Rmax^{n+1}$ be cones. Then,
$(\cV_1,\cV_2)$ is a joined pair of conical hemispaces if and only
if the following statements hold:
\begin{equation}\label{e:sectnonzero}
C_{\cV_1}^{\alpha}\cap C_{\cV_2}^{\alpha}=\emptyset\; \text{ and } \; C_{\cV_1}^{\alpha}\cup C_{\cV_2}^{\alpha}=\Rmax^n\ \text{for all } \alpha\neq\bzero ,
\end{equation}
\begin{equation}\label{e:sectzero}
C_{\cV_1}^{\bzero}\cap C_{\cV_2}^{\bzero}=\{\bzero\} \; \text{ and } \; C_{\cV_1}^{\bzero}\cup C_{\cV_2}^{\bzero}=\Rmax^n .
\end{equation}
\end{lemma}

\begin{proof}
Assume that $(\cV_1,\cV_2)$ is a joined pair of conical hemispaces,
i.e. $\cV_1\cup \cV_2=\Rmax^{n+1}$ and $\cV_1\cap \cV_2=\{\bzero
\}$.

Let $\alpha\in\Rmax$. Then, given any $x\in\Rmax^n$, since $\cV_1\cup \cV_2=\Rmax^{n+1}$
we have $(x,\alpha)\in\cV_1 \cup \cV_2$, and so $x\in C_{\cV_1}^{\alpha}\cup C_{\cV_2}^{\alpha}$. Since $x\in\Rmax^n$ is arbitrary, this shows $C_{\cV_1}^{\alpha}\cup C_{\cV_2}^{\alpha}=\Rmax^n$ for any $\alpha\in\Rmax$.

Suppose now that $\alpha\neq\bzero$ and $C_{\cV_1}^{\alpha}\cap C_{\cV_2}^{\alpha}\neq \emptyset$. Let
$x\in C_{\cV_1}^{\alpha}\cap C_{\cV_2}^{\alpha}$. Then, we have $(x,\alpha)\in \cV_1\cap \cV_2$, which contradicts the fact that $\cV_1\cap \cV_2=\{\bzero \}$ because $\alpha\neq\bzero$. This proves that
$C_{\cV_1}^{\alpha}\cap C_{\cV_2}^{\alpha}= \emptyset$ for $\alpha\neq\bzero$.

Since $\bzero\in \cV_1\cap \cV_2$, we have $\bzero\in C_{\cV_1}^{\bzero}\cap C_{\cV_2}^{\bzero}$. Furthermore, if for $x\neq \bzero$ we had $x\in C_{\cV_1}^{\bzero}\cap C_{\cV_2}^{\bzero}$, then
the non-null vector $(x,\bzero)$ would belong to $\cV_1\cap \cV_2$, contradicting the fact that $\cV_1\cap \cV_2=\{\bzero \}$. This shows that $C_{\cV_1}^{\bzero}\cap C_{\cV_2}^{\bzero}=\{\bzero\}$,
and completes the proof of~\eqref{e:sectnonzero} and~\eqref{e:sectzero}.

Assume now that~\eqref{e:sectnonzero} and~\eqref{e:sectzero} are satisfied.

Given any $x\in \Rmax^n$ and $\alpha \in \Rmax$, since $C_{\cV_1}^{\alpha}\cup C_{\cV_2}^{\alpha}=\Rmax^n$,
we have $x\in C_{\cV_1}^{\alpha}\cup C_{\cV_2}^{\alpha}$. It follows that
$(x,\alpha)\in \cV_1\cup \cV_2$. Since $x\in \Rmax^n$ and $\alpha \in \Rmax$ are arbitrary,
we conclude that $\cV_1\cup \cV_2=\Rmax^{n+1}$.

Finally, let $(x,\alpha)\in \cV_1\cap \cV_2$. Then $x\in
C_{\cV_1}^{\alpha}\cap C_{\cV_2}^{\alpha}$, and
by~\eqref{e:sectnonzero} and~\eqref{e:sectzero} we necessarily have
$x=\bzero$ and $\alpha=\bzero$. This shows that $\cV_1\cap
\cV_2=\{\bzero \}$, and completes the proof of the fact that
$(\cV_1,\cV_2)$ is a joined pair of conical hemispaces.
\end{proof}

The following theorem relates complementary pairs of hemispaces in
$\Rmax^n$ with joined pairs of conical hemispaces in $\Rmax^{n+1}$
through the concept of section.

\begin{theorem}\label{t:hemihomog}
Let $\cH_1,\cH_2\subseteq \Rmax^n$ be a complementary pair of
hemispaces, and let $(P_1,R_1)$ and $(P_2,R_2)$ determine
respectively the $(P,R)$-decompositions of $\cH_1$ and $\cH_2$ given
by Theorem~\ref{t:onetwo}. Then, the cones
\begin{equation}\label{e:cv1}
\cV_1:=\cone \left( \left\{ (x, \bunity) \mid x \in P_1\right\} \cup
\left\{ (y, \bzero) \mid y \in R_1\right\} \right)
\end{equation}
and
\begin{equation}\label{e:cv2}
\cV_2:=\cone \left( \left\{ (x, \bunity) \mid x \in P_2\right\} \cup
\left\{ (y, \bzero) \mid y \in R_2\right\} \right)
\end{equation}
satisfy $\cH_1=C_{\cV_1}^{\bunity}$ and $\cH_2=C_{\cV_2}^{\bunity}$,
and $(\cV_1,\cV_2)$ is a joined pair of conical hemispaces in
$\Rmax^{n+1}$.
\end{theorem}

\begin{proof}
In the first place, observe that by Corollary~\ref{CoroOfp:gen-hom} we have $C_{\cV_1}^{\bunity}=\cH_1$ and
$C_{\cV_2}^{\bunity}=\cH_2$

To prove that $(\cV_1,\cV_2)$ is a joined pair of conical
hemispaces, we show that~\eqref{e:sectnonzero}
and~\eqref{e:sectzero} are satisfied and then use
Lemma~\ref{l:CharaHemiWithSections}.

Let us first prove~\eqref{e:sectnonzero}. Since $(\cH_1,\cH_2)=(C_{\cV_1}^{\bunity},C_{\cV_2}^{\bunity})$ is a complementary pair of hemispaces, it follows that~\eqref{e:sectnonzero} holds for
$\alpha=\bunity$.
For the case of general $\alpha\neq \bzero$, observe that
\begin{equation}\label{e:sectsimple}
C_{\cV_1\cap\cV_2}^{\alpha}=C_{\cV_1}^{\alpha}\cap C_{\cV_2}^{\alpha}\text{ and } C_{\cV_1\cup\cV_2}^{\alpha}=C_{\cV_1}^{\alpha}
\cup C_{\cV_2}^{\alpha}\text{ for all }\alpha\in\Rmax\; .
\end{equation}
Since $\cV_1\cap\cV_2$ and $\cV_1\cup\cV_2$ are closed under multiplication by scalars, using~\eqref{e:sectsimple} and
Proposition~\ref{p:section} we conclude that
\[
\begin{split}
C_{\cV_1}^{\alpha}\cap C_{\cV_2}^{\alpha}&=C_{\cV_1\cap\cV_2}^{\alpha}=
\{\alpha x\mid x\in C_{\cV_1\cap\cV_2}^{\bunity}\}=\{\alpha x\mid x\in C_{\cV_1}^{\bunity}\cap C_{\cV_2}^{\bunity}\}=
\{\alpha x\mid x\in \cH_1\cap\cH_2\}=\emptyset \; ,\\
C_{\cV_1}^{\alpha}\cup C_{\cV_2}^{\alpha}&=C_{\cV_1\cup\cV_2}^{\alpha}
=\{\alpha x\mid x\in C_{\cV_1\cup\cV_2}^{\bunity}\} =\{\alpha x\mid x\in C_{\cV_1}^{\bunity}\cup C_{\cV_2}^{\bunity}\}=
\{\alpha x\mid x\in \cH_1\cup\cH_2\}= \Rmax^n \; .
\end{split}
\]
Thus we obtained~\eqref{e:sectnonzero}.

It remains to prove~\eqref{e:sectzero}. Equations~\eqref{e:cv1}
and~\eqref{e:cv2} imply that $C_{\cV_1}^{\bzero}=\cone(R_1)$ and
$C_{\cV_2}^{\bzero}=\cone(R_2)$, so it remains to show that
$(\cone(R_1),\cone(R_2))$ is a joined pair of conical hemispaces of
$\Rmax^n$.

Let us show first that $\cone(R_1)\cup\cone(R_2)=\Rmax^n$. Take a
vector $z\in \Rmax^n$. As $(\cH_1,\cH_2)$ is a complementary pair of
hemispaces, either $z\in \cH_1$ or $z\in \cH_2$. Assume $z\in\cH_1$.
By Theorem~\ref{t:multiorder2-conv} (taking $\cH_2$ as $\cC$ and
$\cH_1$ as its complement),
it follows that $\Sector_i(z) \subseteq \cH_1$ for $i=n+1$ or for
some $i\in \supp(z)$. If $\Sector_i(z)\subseteq \cH_1$ for some
$i\neq n+1$, then
$z\in\cone(R_1)$ by Lemma~\ref{l:sizspanr}. In the case when
$\Sector_i(z)\not\subseteq\cH_1$ for any $i\neq n+1$, we have
$\Sector_{n+1}(z)\subseteq \cH_1$, and we consider $\alpha z$ for
$\alpha \neq \bzero$.

Suppose that for some $\alpha\neq \bzero$ we have
$\Sector_{n+1}(\alpha z)\not\subseteq\cH_1$ and
$\Sector_{n+1}(\alpha z)\not\subseteq\cH_2$. Then $\Sector_i(\alpha
z)\subseteq\cH_1$ or $\Sector_i(\alpha z)\subseteq\cH_2$  for some
$i\neq n+1$, by Theorem~\ref{t:multiorder2-conv}. If
$\Sector_i(\alpha z)\subseteq\cH_1$ then $z\in\cone(R_1)$, and if
$\Sector_i(\alpha z)\subseteq\cH_2$ then $z\in\cone (R_2)$, by
Lemma~\ref{l:sizspanr}, so $z\in\cone(R_1)\cup\cone(R_2)$.

We are left with the case when $\Sector_{n+1}(\alpha
z)\subseteq\cH_1$ or $\Sector_{n+1}(\alpha z)\subseteq\cH_2$ for
each $\alpha$. Since by Lemma~\ref{l:sn+1incr} the sets
$\Sector_{n+1}(\alpha z)$ are increasing with $\alpha$, it can be
only that either $\Sector_{n+1}(\alpha z)\subseteq\cH_1$ for all
$\alpha$, or $\Sector_{n+1}(\alpha z)\subseteq\cH_2$ for all
$\alpha$. Assume the first case. Then, we obtain that all vectors
$x$ with $\supp (x)\subseteq\supp (z)$ are in $\cH_1$, since
$x\in\Sector_{n+1}(\alpha z)$ with $\alpha=\oplus_{i\in\supp(z)}
x_iz_i^{-1}$ holds for every such $x$. But then
$\Sector_i(z)\subseteq\cH_1$ for any $i\in\supp(z)$, implying that
$z\in\cone (R_1)$.

We have shown that if $z\in\cH_1$ then
$z\in\cone(R_1)\cup\cone(R_2)$. The same statement holds in the case
of $z\in\cH_2$ (by symmetry). Thus
$\cone(R_1)\cup\cone(R_2)=\Rmax^n$ is proved, and it remains to show
that $\cone(R_1)\cap\cone(R_2)=\{\bzero\}$.

Assume by contradiction that $z\in \cone(R_1)\cap\cone(R_2)$ and
$z\neq \bzero$. As $z\in\cone(R_1)$, we have $z=\oplus_{x\in R_1} \beta_x x$,
where only a finite number of the scalars $\beta_x$ are not equal to $\bzero$. Observe that $R_1\neq \emptyset$ and at least
one $\beta_x$ is not equal to $\bzero$ because $z\neq \bzero$.
By~\eqref{e:PRhemi}, $R_1$ is composed of vectors
of the form $\uvector{i}\oplus y_jy_i^{-1}\uvector{j}$, where
$y\in\Rmax^n$ and $i,j\in \supp(y)$ are such that $\Sector_i(y)\subseteq\cH_1$.
Consequently we have $\beta (\uvector{i}\oplus y_{j}y_{i}^{-1}\uvector{j})\leq z$ for some $\beta\in \Rinv$,
$y\in\Rmax^n$ and $i,j\in \supp(y)$ such that $\Sector_{i}(y)\subseteq\cH_1$.
Since $\Sector_{i}(y)\subseteq\cH_1$, by~\eqref{e:PRhemi} it follows that $y_{i}\uvector{i}\in P_1$.
As $z\in\cone(R_2)$, for the same
reasons as above there also exist $\beta'\in \Rinv$,
$y'\in\Rmax^n$ and $i',j'\in \supp(y')$ such that $\beta' (\uvector{i'}\oplus y'_{j'}(y'_{i'})^{-1}\uvector{j'})\leq z$ and $y'_{i'}\uvector{i'}\in P_2$.

If $\lambda \geq (y_{i}\beta ^{-1}\oplus y'_{i'}(\beta')^{-1})$ then
$\lambda \geq y_{i}\beta ^{-1}$, whence using also that $\beta
\uvector{i}\leq z$, we obtain $y_{i}\uvector{i}=y_{i}\beta \beta
^{-1}\uvector{i}\leq \lambda \beta \uvector{i}\leq \lambda z.$
Similarly, since $\lambda \geq y'_{i'} (\beta')^{-1},$ we obtain
$y'_{i'}\uvector{i'}\leq \lambda z$. These inequalities can be
written as equalities $y_{i}\uvector{i}\oplus \lambda
z=y'_{i'}\uvector{i'}\oplus \lambda z=\lambda z$, whence $\lambda
z\in \conv(P_{1})\oplus\cone (R_{1})=\cH_{1}$ and $\lambda
z\in\conv(P_{2})\oplus\cone(R_{2})=\cH_{2}$, so $\lambda z\in
\cH_{1}\cap \cH_{2}$, in contradiction with the assumption
$\cH_{1}\cap \cH_{2}=\emptyset$. Thus the proof
of~\eqref{e:sectnonzero} and~\eqref{e:sectzero} is complete and
Lemma~\ref{l:CharaHemiWithSections} implies that $(\cV_{1},\cV_{2})$
is a joined pair of conical hemispaces.
\end{proof}

\subsection{On the $(P,R)$-decomposition of conical hemispaces}\label{ss:hemicone}

We know that the $(P,R)$-decomposition of a conical hemispace, as a linear span of quasisectors
(Theorem~\ref{t:conhemsem}), consists of unit vectors and linear
combinations of two unit vectors (Theorem~\ref{t:onetwo}).
Therefore, to describe the $(P,R)$-decompositions of a joined pair of conical hemispaces
we need to understand how the linear combinations of
two unit vectors are distributed among them. With this aim, we first
associate with a non-trivial joined pair $(\cV_1,\cV_2)$ of conical
hemispaces in $\Rmax^n$ the index sets
\begin{equation}\label{def:IJ}
I:=\left\{ i\in [n]\mid \uvector{i}\in\cV_1\right\}
\; \makebox{ and } \;
J:=\left\{ j\in [n]\mid \uvector{j}\in\cV_2\right\} .
\end{equation}
The following lemma is elementary and will rather serve to define
below the coefficients $\alpha_{ij}$. In what follows, for some
purposes it will be convenient to assume that scalars can also take
the value $+\infty$ (the structure which is obtained defining
$\lambda \oplus (+\infty):= +\infty, (+\infty)\oplus\lambda:=
+\infty$ for $\lambda \in \Rmax$, $\lambda \otimes (+\infty):=
+\infty, (+\infty)\otimes \lambda:= +\infty$ for $\lambda \in \Rinv$
and $\bzero \otimes (+\infty):= \bzero, (+\infty)  \otimes\bzero :=
\bzero$ is usually known as the completed semifield, see for
instance~\cite{CGQS-05}) and to adopt the convention
\begin{equation}\label{newformula-1008}
\uvector{i}\oplus\lambda \uvector{j}=\uvector{j} \text{ if }\lambda =+\infty \; .
\end{equation}

\begin{lemma}\label{l:duets}
Let $(\cV_1,\cV_2)$ be a non-trivial joined pair of conical
hemispaces of $\Rmax^n$, and let $I,J\subset [n]$ be defined as
in~\eqref{def:IJ}. Then, for any $i\in I$ and $j\in J$ we have
\[
\sup\left\{\alpha \in \Rmax\cup \{+\infty\}\mid \uvector{i}\oplus\alpha \uvector{j}\in\cV_1\right\}=
\inf\left\{\beta \in \Rmax\cup \{+\infty\} \mid \uvector{i}\oplus\beta \uvector{j}\in\cV_2\right\}.
\]
\end{lemma}

\begin{proof}
In the sequel, we will use the fact that every linear combination of two
unit vectors belongs either to $\cV_1$ or to $\cV_2$, which follows
from $\cV_1 \cap \cV_2 = \{\bzero \}$ and $\cV_1 \cup \cV_2 =
\Rmax^n$.

First, assume that $\inf\left\{\beta \in \Rmax\cup \{+\infty\}
\mid \uvector{i}\oplus\beta \uvector{j}\in\cV_2\right\}=+\infty$,
which implies $\uvector{i}\oplus\beta \uvector{j}\not \in\cV_2$ for
all $\beta \in \Rmax$. Then, we have $\uvector{i}\oplus\alpha
\uvector{j}\in\cV_1$ for all $\alpha \in \Rmax$, and so
$\sup\left\{\alpha \in \Rmax\cup \{+\infty\}\mid
\uvector{i}\oplus\alpha \uvector{j}\in\cV_1\right\}= +\infty$.

Assume now that $\inf\left\{\beta \in \Rmax\cup \{+\infty\}
\mid \uvector{i}\oplus\beta \uvector{j}\in\cV_2\right\}\neq +\infty$.
Observe that we have the following implication:
\begin{equation}\label{e:impl}
\uvector{i}\oplus\beta\uvector{j}\in\cV_2\;,\;\gamma\geq\beta\Rightarrow \uvector{i}\oplus\gamma\uvector{j}\in\cV_2\;,
\end{equation}
since $\uvector{j}\in\cV_2$ and, further,
$\uvector{i}\oplus\gamma\uvector{j}=(\uvector{i}\oplus\beta\uvector{j})\oplus\gamma\uvector{j}\in\cV_2$
if $\gamma\geq\beta$. Thus,
\begin{equation}\label{e:supleqinf}
\sup\left\{\alpha \in \Rmax\cup \{+\infty\}\mid \uvector{i}\oplus\alpha \uvector{j}\in\cV_1\right\}\leq
\inf\left\{\beta \in \Rmax\cup \{+\infty\} \mid \uvector{i}\oplus\beta \uvector{j}\in\cV_2\right\}\; ,
\end{equation}
because if we had $>$ in~\eqref{e:supleqinf}, then there would exist $\alpha, \beta \in \Rmax$ with $\alpha >\beta $ such that $\uvector{i}\oplus \alpha \uvector{j}\in \cV_{1}$ and
$\uvector{i}\oplus \beta \uvector{j}\in \cV_{2}$. Then by $\uvector{i}\oplus \beta
\uvector{j}\in \cV_{2}$ and~\eqref{e:impl} it would follow that $\uvector{i}\oplus \alpha \uvector{j}\in \cV_{2}$, whence $\uvector{i}\oplus \alpha \uvector{j}\notin \cV_{1}$, a
contradiction.
If $\inf\left\{\beta \in \Rmax\cup \{+\infty\} \mid
\uvector{i}\oplus\beta
 \uvector{j}\in\cV_2\right\}=\bzero$,
 then the lemma follows from~\eqref{e:supleqinf}.
Thus, it remains to consider the case
$\inf\left\{\beta \in \Rmax\cup \{+\infty\} \mid \uvector{i}\oplus\beta \uvector{j}\in\cV_2\right\}\in \Rinv$.
In this case,
by the definition of $\inf$ we have $\uvector{i}\oplus\alpha \uvector{j}\not \in\cV_2$ for all
$\alpha<\inf\left\{\beta \in \Rmax\cup \{+\infty\} \mid \uvector{i}\oplus\beta \uvector{j}\in\cV_2\right\}$.
Then, since every linear combination of two
unit vectors belongs either to $\cV_1$ or to $\cV_2$, we have
$\uvector{i}\oplus\alpha \uvector{j}\in \cV_1$ for all
$\alpha<\inf\left\{\beta \in \Rmax\cup \{+\infty\} \mid \uvector{i}\oplus\beta \uvector{j}\in\cV_2\right\}$,
and so~\eqref{e:supleqinf} must be satisfied with equality. This completes the proof.
\end{proof}

Henceforth, the matrix whose entries are the coefficients
\begin{equation}\label{alphaij-conic}
\alpha_{ij}:=\sup\left\{\alpha \in \Rmax\cup \{+\infty\}\mid \uvector{i}\oplus\alpha \uvector{j}\in\cV_1\right\}=\inf\left\{\beta \in \Rmax\cup \{+\infty\} \mid \uvector{i}\oplus\beta \uvector{j}\in\cV_2\right\}
\end{equation}
will be referred to as the {\em $\alpha$-matrix} (associated with
the non-trivial joined pair $(\cV_1,\cV_2)$ of conical hemispaces).
Besides, with each coefficient $\alpha_{ij}$ we associate the pair
of subsets of $\Rmax \cup \{+\infty\}$ defined by
\begin{equation}\label{alphapmdef}
(\alpha_{ij}^{(-)},\alpha_{ij}^{(+)}):=
\begin{cases}
(\{\lambda \mid \lambda <\alpha_{ij}\},\{\lambda \mid \lambda \geq\alpha_{ij}\}) & \text{ if } \alpha_{ij}\in\Rinv, \uvector{i}\oplus \alpha_{ij} \uvector{j} \in \cV_2,\\
(\{\lambda \mid \lambda \leq\alpha_{ij}\},\{\lambda \mid \lambda >\alpha_{ij}\}) & \text{ if } \alpha_{ij}\in\Rinv, \uvector{i}\oplus \alpha_{ij} \uvector{j} \in \cV_1,\\
(\{\alpha_{ij}\},\{\lambda \mid \lambda >\alpha_{ij}\}) & \text{ if } \alpha_{ij}=\bzero,\\
(\{\lambda \mid \lambda <\alpha_{ij}\},\{\alpha_{ij} \}) & \text{ if } \alpha_{ij}=+\infty .
\end{cases}
\end{equation}
Thus, by Lemma~\ref{l:duets} it follows that
\begin{equation}\label{alphapmdef2}
\left\{\uvector{i}\oplus \lambda \uvector{j}\mid \lambda \in \alpha_{ij}^{(-)}\right\} \subset \cV_1 \makebox{ and }
\left\{\uvector{i}\oplus \lambda \uvector{j}\mid \lambda \in \alpha_{ij}^{(+)}\right\} \subset \cV_2
\end{equation}
for any $i\in I$ and $j\in J$.

Since $\alpha_{ij}^{(+)}\subseteq\Rinv \cup\{+\infty\}$ and
$\alpha_{ij}^{(-)}\subseteq \Rinv \cup \{ \bzero \}$,
observe that the sets $\alpha_{i_1j_1}^{(+)}$ and $\alpha_{i_2j_2}^{(+)}$,
as well as  $\alpha_{i_1j_1}^{(-)}$ and $\alpha_{i_2j_2}^{(-)}$,
can be unambiguously multiplied (by definition,
the product of two sets consists of
all possible products of an element of one set by an element of
the other set)
for any $i_1,i_2\in I$ and $j_1,j_2\in J$.

In the sequel, we write $I^1+\cdots+I^m=I$ if $I^k$ for $k\in [m]$ and
$I$ are index sets such that $I^1 \cup\cdots\cup I^m=I$ and
$I^1,\ldots, I^m$ are pairwise disjoint.

We now formulate one of the main results of the paper: a
characterization of conical hemispaces in terms of
their generators. We will immediately prove that any conical hemispaces fulfils the given conditions. The proof
that these conditions are also sufficient is going to occupy the
remaining part of this section.

\begin{theorem}\label{t:Ghemi}
A non-trivial cone $\cV\subset \Rmax^n$ is a conical hemispace if and only if
\begin{equation}\label{DefHemiWithGen}
\cV=\cone \left( \left\{\uvector{i}\oplus \lambda \uvector{j}\mid i\in I, j\in J, \lambda \in \sigma_{ij}^{(-)}\right\} \right),
\end{equation}
where $I$ is a non-empty proper subset of $[n]$, $J=[n]\setminus I$,
and the sets $\sigma_{ij}^{(-)}$, which are non-empty proper subsets of $\Rmax\cup \{+\infty\}$
either of the form
$\{\lambda \in \Rmax \mid \lambda \leq \sigma_{ij}\}$ or
$\{\lambda \in \Rmax \mid \lambda < \sigma_{ij}\}$ with
$\sigma_{ij}\in \Rmax\cup \{+\infty\}$,
are such that the pairs $(\sigma^{(-)}_{ij},\sigma^{(+)}_{ij})$,
with $\sigma^{(+)}_{ij}$ defined by $\sigma^{(+)}_{ij}:=(\Rmax\cup \{+\infty\})\setminus \sigma^{(-)}_{ij}$, satisfy
\begin{equation}\label{e:rankone}
\sigma_{i_1j_2}^{(+)}\sigma_{i_2j_1}^{(+)}\cap \sigma_{i_1j_1}^{(-)}\sigma_{i_2j_2}^{(-)}=\emptyset \; \makebox{ and } \;
\sigma_{i_1j_2}^{(-)}\sigma_{i_2j_1}^{(-)}\cap \sigma_{i_1j_1}^{(+)}\sigma_{i_2j_2}^{(+)}=\emptyset
\end{equation}
for any $i_1,i_2\in I$ and $j_1,j_2\in J$.
\end{theorem}

\noindent {\em Proof of the ``only if'' part of Theorem~\ref{t:Ghemi}.}
Define $\cV_1:=\cV$ and $\cV_2:=\complement \cV \cup \{ \bzero \}$.
Thus, $(\cV_1,\cV_2)$ is a non-trivial joined pair of conical hemispaces in $\Rmax^n$ because
$\cV$ is a conical hemispace and non-trivial. Let $I$ and $J$ be the
sets defined in~\eqref{def:IJ}. Then, $I$ and $J$ satisfy $J=[n]\setminus I$, and these sets
are non-empty since $(\cV_1,\cV_2)$ is non-trivial.
For $i\in I$ and $j\in J$, let $\sigma_{ij}:=\alpha_{ij}$
and $(\sigma_{ij}^{(-)},\sigma_{ij}^{(+)}):=(\alpha_{ij}^{(-)},\alpha_{ij}^{(+)})$,
where the scalars $\alpha_{ij}$ and the pairs of sets $(\alpha_{ij}^{(-)},\alpha_{ij}^{(+)})$
are defined by~\eqref{alphaij-conic} and~\eqref{alphapmdef} respectively. Then,
the sets $\sigma_{ij}^{(-)}$ and $\sigma_{ij}^{(+)}$ are of the required form.

We claim that
\begin{equation}\label{hemi-grep}
\begin{split}
\cV_1=\cone \left( \left\{\uvector{i}\oplus \lambda \uvector{j}\mid i\in I, j\in J, \lambda \in \sigma_{ij}^{(-)}\right\} \right), \\
\cV_2=\cone \left( \left\{\uvector{i}\oplus \lambda \uvector{j}\mid i\in I, j\in J, \lambda \in \sigma_{ij}^{(+)}\right\} \right).
\end{split}
\end{equation}
Indeed, by Theorem~\ref{t:onetwo} both $\cV_1$ and $\cV_2$
are generated by unit vectors and linear combinations of two unit vectors.
The distribution of unit vectors is given by $I$ and $J$.
Observe that~\eqref{hemi-grep} conforms to this distribution,
since for any $i\in I$, $\uvector{i}$ belongs to the generators
of $\cV_1$ as $\bzero \in \sigma_{ij}^{(-)}$,
and for any $j\in J$, $\uvector{j}$ belongs to the generators of $\cV_2$
since $+\infty \in \sigma_{ij}^{(+)}$. This obviously implies that no linear combination of $\uvector{i_1}$
and $\uvector{i_2}$ with $i_1,i_2\in I$
(resp. of $\uvector{j_1}$ and $\uvector{j_2}$ with $j_1,j_2\in J$)
is necessary in~\eqref{hemi-grep} to generate $\cV_1$ (resp. $\cV_2$).
For $i\in I$ and $j\in J$, the distribution of the linear combinations of
$\uvector{i}$ and $\uvector{j}$ is given by~\eqref{alphapmdef2}. Since $(\sigma_{ij}^{(-)},\sigma_{ij}^{(+)})=(\alpha_{ij}^{(-)}\alpha_{ij}^{(+)})$,
it follows that~\eqref{hemi-grep} also conforms to this distribution.
These observations yield~\eqref{hemi-grep}.

It remains to prove~\eqref{e:rankone}. Assume that
\[
\sigma_{i_1j_2}^{(+)}\sigma_{i_2j_1}^{(+)}\cap \sigma_{i_1j_1}^{(-)}\sigma_{i_2j_2}^{(-)}\neq\emptyset \; .
\]
Then, there exist $\beta_{i_1j_2}\in \sigma_{i_1j_2}^{(+)}$,
$\beta_{i_2j_1}\in \sigma_{i_2j_1}^{(+)}$, $\gamma_{i_1j_1}\in \sigma_{i_1j_1}^{(-)}$
and $\gamma_{i_2j_2}\in \sigma_{i_2j_2}^{(-)}$
such that $\beta_{i_1j_2}\beta_{i_2j_1}=\gamma_{i_1j_1}\gamma_{i_2j_2}$.
For this to hold, the products $\beta_{i_1j_2}\beta_{i_2j_1}$ and
$\gamma_{i_1j_1}\gamma_{i_2j_2}$ should be in $\Rinv$,
and hence $\beta_{i_1j_2}$, $\beta_{i_2j_1}$,
$\gamma_{i_1j_1}$ and $\gamma_{i_2j_2}$ should be in $\Rinv$.
Then, we make the linear combination
\[
z=\uvector{i_1}\oplus \beta_{i_1j_2}\uvector{j_2}\oplus\lambda(\uvector{i_2}\oplus \beta_{i_2j_1} \uvector{j_1})\in\cV_2 \; ,
\]
where $\lambda$ satisfies $\lambda \beta_{i_2j_1}=\gamma_{i_1j_1}$,
hence also $\lambda \gamma_{i_2j_2}=\beta_{i_1j_2}$, and observe that
\[
z=\uvector{i_1}\oplus \gamma_{i_1j_1}\uvector{j_1}\oplus\lambda(\uvector{i_2}\oplus \gamma_{i_2j_2} \uvector{j_2})\in\cV_1 \; .
\]
Thus $\cV_1\cap\cV_2\neq\{\bzero\}$, a contradiction.
{\em This completes the proof of the ``only if'' part of Theorem~\ref{t:Ghemi}. The ``if'' part will be proved later (formally after Remark~\ref{JustRemark}, but the preparations for this proof
will start right after Corollary~\ref{c:rankone}).}

The following result shows
that if a non-trivial cone $\cV$ defined as in~\eqref{DefHemiWithGen} is a conical hemispace,
then $\complement\cV\cup\{\bzero\}$ can be defined as $\cV_2$ in~\eqref{hemi-grep} and the scalars $\sigma_{ij}$ are precisely the entries of the $\alpha$-matrix
associated with the non-trivial joined pair of conical hemispaces $(\cV,\complement\cV\cup\{\bzero\})$.

\begin{proposition}\label{p:complement}
Assume that
\begin{equation}\label{DefHemi}
\cV_1=\cone \left( \left\{\uvector{i}\oplus \lambda \uvector{j}\mid i\in I, j\in J, \lambda \in \sigma_{ij}^{(-)}\right\} \right)
\end{equation}
is a conical hemispace, where $I$ is a non-empty proper subset of $[n]$, $J=[n]\setminus I$,
and for $i\in I$ and $j\in J$ the sets $\sigma_{ij}^{(-)}$ are non-empty proper subsets of
$\Rmax\cup \{+\infty\}$ either of the form
$\{\lambda \in \Rmax \mid \lambda \leq \sigma_{ij}\}$ or
$\{\lambda \in \Rmax \mid \lambda < \sigma_{ij}\}$ with
$\sigma_{ij}\in \Rmax\cup \{+\infty\}$. Then, $\cV_1$ and $\cV_2$ defined by
\begin{equation}\label{DefHemiOther}
\cV_2:=\cone
\left( \left\{\uvector{i}\oplus \lambda \uvector{j}\mid i\in I, j\in J, \lambda \in \sigma_{ij}^{(+)}\right\} \right),
\end{equation}
where $\sigma^{(+)}_{ij}:=(\Rmax\cup \{+\infty\})\setminus
\sigma^{(-)}_{ij}$, form a joined pair of conical hemispaces, and we have $\sigma_{ij}=\alpha_{ij}$ for all $i\in I$ and $j\in J$ with
$\alpha_{ij}$ defined by~\eqref{alphaij-conic}.
\end{proposition}

\begin{proof}
Let $R:=\left\{\uvector{i}\oplus \lambda \uvector{j}\mid i\in I, j\in J, \lambda \in \sigma_{ij}^{(-)}\right\}$. We first claim that the unit vectors
and linear combinations of two unit vectors contained in $\cV_1$ are precisely the ones in $R$.
Indeed given $j\in J$, since $\uvector{j}\not \in R$, it readily follows that $\uvector{j}\not \in \cV_1$. Then, the unit vectors contained in $\cV_1$ are precisely the ones in $R$ (i.e. $\uvector{i}$ for $i\in I$).
Assume now that $\uvector{i}\oplus \beta \uvector{j}\in \cV_1$ for some $i\in I$, $j\in J$ and $\beta \in \Rinv$. Then, we have
$\uvector{i}\oplus \beta \uvector{j}=\oplus_{y\in R} \delta_y y$, where only a finite number of the scalars $\delta_y$ is not equal to $\bzero$. Observe that
\begin{equation}\label{PropSimple}
\begin{split}
\delta_y\neq \bzero & \implies y_k=\bzero \text{ for }k\in [n]\setminus \{i,j\} \implies y=\uvector{i}\oplus \lambda \uvector{j} \text{ for some } \lambda \in \sigma_{ij}^{(-)} \\
& \implies y_i=\bzero , y_j\in \sigma_{ij}^{(-)}, \text{ and } y_k=\bzero \text{ for }k\in [n]\setminus \{i,j\}
\; .
\end{split}
\end{equation}
Then $\bunity=(\uvector{i}\oplus \beta \uvector{j})_i=(\oplus_{y\in R} \delta_y y)_i=\oplus_{y\in R} \delta_y y_i=\oplus_{y\in R} \delta_y$, and so $\delta_y\leq \bunity$ for all $y\in R$.
Besides, since only a finite number of the scalars $\delta_y$ is not equal to $\bzero$ and  $\beta=(\uvector{i}\oplus \beta \uvector{j})_j=(\oplus_{y\in R} \delta_y y)_j=\oplus_{y\in R} \delta_y y_j$,
we conclude that $\beta=\delta_y y_j$ for some $y\in R$ such that $\delta_y \neq \bzero$. Using~\eqref{PropSimple} and the fact that $\lambda \in \sigma_{ij}^{(-)}$ and $\delta \leq \bunity$ imply
$\delta \lambda \in \sigma_{ij}^{(-)}$, it follows that $\beta \in \sigma_{ij}^{(-)}$, and so $\uvector{i}\oplus \beta \uvector{j}\in R$. This completes the proof of our claim.

By Theorem~\ref{t:onetwo}, the conical hemispace $\complement \cV_1\cup \{\bzero \}$
is generated by the unit vectors and linear combinations of two unit vectors which it contains, i.e. those which do not belong to $\cV_1$. By the first part of the proof and the definition of $\sigma^{(+)}_{ij}$ as complements of $\sigma^{(-)}_{ij}$ in $\Rmax\cup \{+\infty\}$, we know that these vectors are precisely the generators of $\cV_2$ in~\eqref{DefHemiOther}.
Then $\cV_2=\complement \cV_1\cup \{\bzero \}$, and so $\cV_1$ and $\cV_2$ form a joined pair of conical hemispaces.

Finally, the fact that the entries $\alpha_{ij}$ of the $\alpha$-matrix associated with $(\cV_1,\cV_2)$  coincide with the scalars $\sigma_{ij}$ follows from their definition~\eqref{alphaij-conic} and
from~\eqref{DefHemi} and~\eqref{DefHemiOther}.
\end{proof}

Condition~\eqref{e:rankone} will be called the {\em rank-one condition},
due to the following observation.

\begin{corollary}\label{c:rankone}
If condition~\eqref{e:rankone} is satisfied and
$\sigma_{ij}\in\Rinv$ for $i\in \{i_1,i_2\}$ and $j\in\{j_1,j_2\}$,
then
$\sigma_{i_1j_1}\sigma_{i_2j_2}=\sigma_{i_1j_2}\sigma_{i_2j_1}$.
In particular, if all the entries of an $\alpha$-matrix belong to $\Rinv$, then it has rank one.
\end{corollary}

In the rest of this subsection, we assume that $I$ is a non-empty
proper subset of $[n]$ and $\cV$ is the non-trivial cone defined
by~\eqref{DefHemiWithGen}, where $J:=[n]\setminus I$ and the sets
$\sigma_{ij}^{(-)}$, which are either of the form $\{\lambda \in
\Rmax \mid \lambda \leq\sigma_{ij}\}$ or $\{\lambda \in \Rmax \mid
\lambda <\sigma_{ij}\}$ with $\sigma_{ij}\in \Rmax\cup \{+\infty\}$,
are such that the pairs $(\sigma^{(-)}_{ij},\sigma^{(+)}_{ij})$,
with $\sigma^{(+)}_{ij}$ defined by
$\sigma^{(+)}_{ij}:=(\Rmax\cup \{+\infty\})\setminus \sigma^{(-)}_{ij}$,
satisfy the rank-one condition~\eqref{e:rankone}. With the objective of showing that any
such cone is a conical hemispace, we first give a detailed
description of the ``thin structure'' of the corresponding
$\sigma$-matrix that follows from the rank-one
condition~\eqref{e:rankone}. This description can be also seen as
one of the main results.

\begin{proposition}\label{p:RD}
If we define
\[
\begin{split}
J_i^< := & \{j\in J\mid \sigma_{ij}\in\Rinv \;\makebox{and}\;\sigma_{ij}\in\sigma_{ij}^{(+)} \}, \\
J_i^{\leq} := & \{j\in J\mid \sigma_{ij}\in\Rinv \;\makebox{and}\;\sigma_{ij}\in\sigma_{ij}^{(-)} \},\\
J_i^{\bzero}:= & \{j\in J\mid  \sigma_{ij}=\bzero\},\\
J_i^{\infty}:= & \{j\in J\mid  \sigma_{ij}=+\infty\},
\end{split}
\]
 for $i\in I$, then by the rank-one condition~\eqref{e:rankone} it follows that:
\begin{enumerate}[(i)]
\item\label{p:RD:P2} $J_i^{\leq}+J_i^<+J_i^{\infty}+J_i^{\bzero}=J$ for each $i\in I$;
\item\label{p:RD:P3}  $J_{i_1}^{\infty}\subseteq J_{i_2}^{\infty}$ or $J_{i_2}^{\infty}\subseteq J_{i_1}^{\infty}$,
and $J_{i_1}^{\bzero}\subseteq J_{i_2}^{\bzero}$ or $J_{i_2}^{\bzero}\subseteq J_{i_1}^{\bzero}$
for any $i_1,i_2\in I$;
\item\label{p:RD:P4} If $(J_{i_1}^{<}\cup J_{i_1}^{\leq})\cap (J_{i_2}^{<}\cup J_{i_2}^{\leq})\neq \emptyset$, then $J_{i_1}^{<}\cup J_{i_1}^{\leq}= J_{i_2}^{<}\cup J_{i_2}^{\leq}$, \break $J_{i_1}^{\infty} = J_{i_2}^{\infty}$ and $J_{i_1}^{\bzero} = J_{i_2}^{\bzero}$;
\item\label{p:RD:P5} If $(J_{i_1}^{<}\cup J_{i_1}^{\leq})\cap (J_{i_2}^{<}\cup J_{i_2}^{\leq})\neq \emptyset$, then $J_{i_1}^{<}\subseteq J_{i_2}^{<}$ or $J_{i_2}^{<}\subseteq J_{i_1}^{<}$;
\item\label{p:RD:P6} If $(J_{i_1}^{<}\cup J_{i_1}^{\leq})\cap (J_{i_2}^{<}\cup J_{i_2}^{\leq})\neq \emptyset$, then
there exists $\lambda \in \Rinv$ such that \break $\sigma_{i_1
j}=\lambda \sigma_{i_2 j}$ for all $j\in J_{i_1}^{<} \cup
J_{i_1}^{\leq}= J_{i_2}^{<}\cup J_{i_2}^{\leq}$.
\end{enumerate}
\end{proposition}

\begin{proof}
In this proof, we will use $F$, $\geq F$ and $\leq F$
to represent an entry of a matrix which belongs to $\Rinv$,
$\Rinv \cup \{+\infty\}$ and $\Rinv \cup \{\bzero\} =\Rmax$, respectively.

\eqref{p:RD:P2} This property readily follows from the definition of the sets
$J_i^<$, $J_i^{\leq}$, $J_i^{\bzero}$, and $J_i^{\infty}$.

\eqref{p:RD:P3} If these conditions are violated, then the
$\sigma$-matrix has one of the following $2\times 2$ minors
\[
\begin{pmatrix}
+\infty & \leq F\\
\leq F & +\infty
\end{pmatrix}
\; ,\quad
\begin{pmatrix}
\bzero & \geq F\\
\geq F & \bzero
\end{pmatrix}\; ,
\]
violating~\eqref{e:rankone}.

\eqref{p:RD:P4} If this condition is violated, then the
$\sigma$-matrix has one of the following $2\times 2$ minors
\[
\begin{pmatrix}
F & F\\
\bzero & F
\end{pmatrix},\quad
\begin{pmatrix}
F & F\\
+\infty & F
\end{pmatrix},\quad
\begin{pmatrix}
+\infty & F\\
\bzero & F
\end{pmatrix},
\]
violating~\eqref{e:rankone}. More precisely, one of the first two minors will appear when
$(J_{i_1}^{<}\cup J_{i_1}^{\leq})\cap (J_{i_2}^{<}\cup J_{i_2}^{\leq})\neq\emptyset$ but
$(J_{i_1}^{<}\cup J_{i_1}^{\leq})\neq (J_{i_2}^{<}\cup J_{i_2}^{\leq})$. The third one will appear if
$(J_{i_1}^{<}\cup J_{i_1}^{\leq})=(J_{i_2}^{<}\cup J_{i_2}^{\leq})\neq\emptyset$ but
$J_{i_1}^{\infty}\neq J_{i_2}^{\infty}$ (equivalently, $J_{i_1}^{\bzero}\neq J_{i_2}^{\bzero}$).

\eqref{p:RD:P5} If $J_{i_1}^{<}\subseteq J_{i_2}^{<}$ and
$J_{i_2}^{<}\subseteq J_{i_1}^{<}$ do not hold for some $i_1,i_2$,
then there exist $j_1$ and $j_2$ such that
$\sigma_{i_1j_1}\in\sigma_{i_1j_1}^{(+)}$,
$\sigma_{i_2j_2}\in\sigma_{i_2j_2}^{(+)}$,
$\sigma_{i_1j_2}\in\sigma_{i_1j_2}^{(-)}$,
$\sigma_{i_2j_1}\in\sigma_{i_2j_1}^{(-)}$, and
$\sigma_{i_1j_1},\sigma_{i_1j_2},\sigma_{i_2j_1},\sigma_{i_2,j_2}\in\Rinv$.
However, this contradicts the rank-one condition~\eqref{e:rankone},
since
$\sigma_{i_1j_1}\sigma_{i_2j_2}=\sigma_{i_1j_2}\sigma_{i_2j_1}$ by
Corollary~\ref{c:rankone}.

\eqref{p:RD:P6} This property follows from Corollary~\ref{c:rankone}
and Property~\eqref{p:RD:P4}.
\end{proof}

\begin{remark}\label{r:longeriii}
Regarding Property~\eqref{p:RD:P3} of Proposition~\ref{p:RD},
observe that the condition ``$J_{i_1}^{\infty}\subseteq J_{i_2}^{\infty}$ or $J_{i_2}^{\infty}\subseteq J_{i_1}^{\infty}$'' can be equivalently formulated as
``$J_{i_1}^{<}\cup J_{i_1}^{\leq}\cup J_{i_1}^{\bzero}\subseteq J_{i_2}^{<}\cup J_{i_2}^{\leq}\cup J_{i_2}^{\bzero}$ or $J_{i_2}^{<}\cup J_{i_2}^{\leq}\cup J_{i_2}^{\bzero}\subseteq J_{i_1}^{<}\cup J_{i_1}^{\leq}\cup J_{i_1}^{\bzero}$'' for any $i_1,i_2\in I$.
Similarly, the condition ``$J_{i_1}^{\bzero}\subseteq J_{i_2}^{\bzero}$ or $J_{i_2}^{\bzero}\subseteq J_{i_1}^{\bzero}$'' can be equivalently formulated as
``$J_{i_1}^{<}\cup J_{i_1}^{\leq}\cup J_{i_1}^{\infty}\subseteq J_{i_2}^{<}\cup J_{i_2}^{\leq}\cup J_{i_2}^{\infty}$ or $J_{i_2}^{<}\cup J_{i_2}^{\leq}\cup J_{i_2}^{\infty}\subseteq J_{i_1}^{<}\cup J_{i_1}^{\leq}\cup J_{i_1}^{\infty}$'' for any $i_1,i_2\in I$.
\end{remark}

Consider the equivalence relation on $I$ defined by
\[
i_1 \sim i_2 \iff J_{i_1}^{\infty}=J_{i_2}^{\infty}
\makebox{ and } J_{i_1}^{\bzero}=J_{i_2}^{\bzero}.
\]
By Proposition~\ref{p:RD} part~\eqref{p:RD:P3}, the relation
\[
i_1\preceq i_2 \iff
\begin{cases}
J_{i_2}^{\infty} \subset J_{i_1}^{\infty} \makebox{ or }\\
J_{i_2}^{\infty} = J_{i_1}^{\infty} \makebox{ and } J_{i_1}^{\bzero} \subseteq J_{i_2}^{\bzero}
\end{cases}
\]
defines a total order on $I$, which induces a total order
(also denoted by $\preceq$) on the equivalence classes associated with $\sim$.
Assume that $I^1, \ldots , I^p$ are these equivalence classes and that
$I^1 \preceq I^2 \preceq \cdots \preceq I^p$.

By definition,
note that there exist subsets $L^1,\ldots,L^p$, $K^1,\ldots ,K^p$,
and $J^1,\ldots ,J^p$ of $J$, such that $J_i^{\bzero}=L^r$,
$J_{i}^{\infty}=K^r$ and $J_{i}^{<}\cup J_{i}^{\leq}=J^r$ for $i\in I^r$. Thus,
by Proposition~\ref{p:RD} part~\eqref{p:RD:P2}, it follows that
\[
J^r+K^r+L^r=J
\]
for $r\in [p]$, and from part~\eqref{p:RD:P4}
we conclude that the sets $J^1,\ldots ,J^p$
are pairwise disjoint. Moreover, for $r\in [2,p]$ we have
\begin{equation}\label{PropertyJrcupKrsubsetKr-1}
J^r \cup K^r \subseteq K^{r-1} ,
\end{equation}
or equivalently
\[
J^{r-1}\cup L^{r-1}\subseteq L^r.
\]
Indeed, if $i_1\in I^{r-1}$ and $i_2\in I^r$,
using Remark~\ref{r:longeriii} we conclude that
either $J_{i_2}^<\cup J_{i_2}^{\leq}\cup J_{i_2}^{\infty}\subseteq
J_{i_1}^<\cup J_{i_1}^{\leq}\cup J_{i_1}^{\infty}$ or $J_{i_1}^<\cup J_{i_1}^{\leq}\cup J_{i_1}^{\infty}\subseteq J_{i_2}^<\cup J_{i_2}^{\leq}\cup J_{i_2}^{\infty}$.
Using Proposition~\ref{p:RD} part~\eqref{p:RD:P3} and the fact that
$J^{r-1}=J_{i_1}^<\cup J_{i_1}^{\leq}$ and $J^r=J_{i_2}^<\cup J_{i_2}^{\leq}$
are disjoint, it follows that either
$J_{i_2}^<\cup J_{i_2}^{\leq}\cup J_{i_2}^{\infty}\subseteq J_{i_1}^{\infty}$ or
$J_{i_1}^<\cup J_{i_1}^{\leq}\cup J_{i_1}^{\infty}\subseteq J_{i_2}^{\infty}$.
In the former case, we have $J^r\cup K^r=J_{i_2}^<\cup J_{i_2}^{\leq}\cup J_{i_2}^{\infty} \subseteq  J_{i_1}^{\infty}=K^{r-1}$.
In the latter case, as  $i_1 \preceq i_2$,
we have $J_{i_2}^{\infty}\subseteq J_{i_1}^{\infty}$ and so
$K^{r-1}=J_{i_1}^{\infty}=J_{i_2}^{\infty}=K^r$
and $J^{r-1}= J_{i_1}^<\cup J_{i_1}^{\leq}=\emptyset$.
Thus, $L^{r-1}=J_{i_1}^{\bzero}\subseteq J_{i_2}^{\bzero}=L^r$
because $i_1 \preceq i_2$,
which implies $J^r\cup K^r=J\setminus L^r\subseteq J\setminus L^{r-1}=J^{r-1}\cup K^{r-1}= K^{r-1}$.

Finally, note that by Proposition~\ref{p:RD} part~\eqref{p:RD:P5}, we have
\begin{equation}\label{PropertyChain}
J_{i_1}^<\subseteq J_{i_2}^< \;\makebox{ or }\;
J_{i_2}^<\subseteq J_{i_1}^<
\end{equation}
for all $i_1,i_2\in I^r$ and $r\in [p]$.

Observe that $\cV$ is also generated by the set
\[
\bigcup_{i\in I}  \Bigl( \left\{ \uvector{i}\right\} \cup \left\{\uvector{i}\oplus \sigma_{ij} \uvector{j} \mid j\in J_i^{\leq}\right\}
 \cup  \left\{ \uvector{i}\oplus \lambda \uvector{j} \mid j\in J_i^{<} , \lambda < \sigma_{ij}\right\} \cup \left\{ \uvector{i}\oplus \lambda \uvector{j} \mid j\in J_i^{\infty} , \lambda \in \Rinv  \right\} \Bigr) \; ,
\]
since any vector of the form $\uvector{i}\oplus \lambda
\uvector{j}$, where $j\in J_i^{\leq}$ and $\lambda< \sigma_{ij}$,
can be expressed as a linear combination of $\uvector{i}\oplus
\sigma_{ij} \uvector{j}$ and $\uvector{i}$. Moreover, defining
\begin{equation}\label{def:CiDi}
\begin{split}
\cC_i&:=\cone\left(\left\{ \uvector{i}\right\}\cup \left\{\uvector{i}\oplus \sigma_{ij} \uvector{j} \mid j\in J_i^{\leq}\right\} \cup  \left\{ \uvector{i}\oplus \lambda \uvector{j} \mid j\in J_i^{<} , \lambda < \sigma_{ij}\right\}\right),\\
\cD_i&:=\cone\left(\left\{ \uvector{i}\right\}  \cup \left\{ \uvector{i}\oplus \lambda \uvector{j} \mid j\in J_i^{\infty} , \lambda \in \Rinv  \right\} \right),
\end{split}
\end{equation}
for $i\in I$, we have $\cV=\mysup_{i\in I} (\cC_i\oplus \cD_i)$.

\begin{lemma}\label{l:Ci}
There exist $\beta_h \in \Rinv$, for $h\in I$, and
$\gamma_j \in \Rinv$, for $j\in \cup_{i\in I}(J_i^{\leq}\cup J_i^{<})$,
such that for each $i\in I$, the set of non-null vectors of the cone $\cC_i$ is the set
of vectors satisfying
\begin{equation}\label{e:Ci}
\left\{
\begin{array}{l}
\gamma_j x_j\leq \beta_i x_i\;\makebox{ for all } \;j\in J_i^{\leq}  \\
\gamma_j x_j< \beta_i x_i
\;\makebox{ for all } \;j\in J_i^{<} \\
x_j=\bzero\;\makebox{ for all } \;j\in J_i^{\bzero}\cup J_i^{\infty}\cup (I\bez \{i\})
\end{array}
\right.
\end{equation}
\end{lemma}
\begin{proof}
Proposition~\ref{p:RD} part~\eqref{p:RD:P6} implies that there exist
$\beta_i,\gamma_j\in\Rinv$ such that
$\sigma_{ij}=\gamma_j^{-1}\beta_i$ for all $\sigma_{ij}\in\Rinv $.
Thus, the cone $\cC_i$ can be equivalently defined by
\[
\cC_i=\cone\left(\left\{ \uvector{i}\right\}\cup \left\{\gamma_j\uvector{i}\oplus \beta_i \uvector{j} \mid j\in J_i^{\leq}\right\} \cup  \left\{ \gamma_j \uvector{i}\oplus \lambda \beta_i \uvector{j} \mid j\in J_i^{<} , \lambda < \bunity \right\}\right) .
\]
Next, any non-null vector $x\in\cC_i$ can be written as a linear combination of vectors in the cones
\[
\begin{split}
\cC^\leq_{ij}&:=\cone\left(\left\{\uvector{i}\right\}\cup
\left\{ \gamma_j \uvector{i}\oplus \beta_i \uvector{j} \mid j\in J_i^{\leq}\right\}\right) , \\
\cC^<_{ij}&:=\cone\left(\left\{\uvector{i}\right\}\cup  \left\{ \gamma_j \uvector{i}\oplus \lambda\beta_i \uvector{j} \mid j\in J_i^{<},\lambda < \bunity\right\}\right),
\end{split}
\]
with the same coefficient $x_i$ at $\uvector{i}$.
The generators of $\cC^\leq_{ij}$ and $\cC^<_{ij}$ satisfy
the first and second conditions of~\eqref{e:Ci} respectively,
hence $x$ also satisfies all these conditions.
Conversely, each non-null vector $x$ satisfying~\eqref{e:Ci} can be written
(using similar ideas to those in the proof of Proposition~\ref{p:ckigenconv})
as a linear combination of the generators of $\cC^\leq_{ij}$ and $\cC^<_{ij}$,
and so it belongs to $\cC_i$.
\end{proof}

\if{
\begin{proof}
Proposition~\ref{p:RD} part~\eqref{p:RD:P6} implies that there exist
$\beta_i,\gamma_j\in\Rinv$ such that
$\sigma_{ij}=\gamma_j^{-1}\beta_i$ for all $\sigma_{ij}\in\Rinv $.
Thus, the cone $\cC_i$ can be equivalently defined by
\begin{equation}
\label{e:cci}
\cC_i=\cone\left(\left\{ \uvector{i}\right\}\cup \left\{\gamma_j\uvector{i}\oplus \beta_i \uvector{j} \mid j\in J_i^{\leq}\right\} \cup  \left\{ \gamma_j \uvector{i}\oplus \lambda \beta_i \uvector{j} \mid j\in J_i^{<} , \lambda < \bunity \right\}\right) .
\end{equation}
Let $x$ be a non-null vector in $\cC_i$.
All generators of $\cC_i$ in~\eqref{e:cci} satisfy
all conditions of~\eqref{e:Ci}. As the set described by~\eqref{e:Ci} (enlarged with the $\bzero$ vector)
is a cone, $x$ also satisfies all these conditions.
Conversely, each non-null vector $x$ satisfying~\eqref{e:Ci} can be written
(using similar ideas to those in the proof of Proposition~\ref{p:ckigenconv})
as a linear combination of the generators of $\cC^\leq_{ij}$ and $\cC^<_{ij}$,
and so it belongs to $\cC_i$.
\end{proof}
}\fi

Later we will show that certain Minkowski sums of the cones $\cC_i$
are conical hemispaces. To this end, note that $\cC_i = \left\{x \in
\Rmax^n\mid x_j=\bzero \; \makebox{ for} \;j\neq i\right\}$ if
$J_{i}^{<}\cup J_{i}^{\leq}= \emptyset$, and so
\begin{equation}\label{e:hemitriv}
\mysup_{i\in \Tilde{I}} \cC_i = \left\{x \in \Rmax^n\mid x_j=\bzero \; \makebox{ for all} \;j\not \in \Tilde{I}\right\}
\end{equation}
when for $\Tilde{I}\subseteq I$ we have $J_{i}^{<}\cup J_{i}^{\leq}=
\emptyset$ for all $i\in \Tilde{I}$. Evidently, any set given
by~\eqref{e:hemitriv} is a conical hemispace.

\begin{remark}\label{RemarkNullVectorOnly}
Since $\cV = \mysup_{i\in I} (\cC_i\oplus \cD_i)$,
observe that the null vector $\bzero$ is the only vector $x$ in $\cV$
satisfying $x_i = \bzero$ for all $i\in I$.
\end{remark}

\begin{theorem}\label{t:hatx}
Given $x\in \Rmax^n$, if $x_i \neq \bzero$ for some $i\in I$,
let $h:=\min \{r\in [p] \mid x_{t}\neq \bzero \makebox{ for some } t \in I^r \}$
and $\hat{x}\in \Rmax^n$ be the vector defined by
$\hat{x}_k:=\bzero$ if $k\in \left(\cup_{r>h} I^r\right) \cup K^h$
and $\hat{x}_k:=x_k$ otherwise.
Then, $x\in \cV$ if and only if $\hat{x}\in \mysup_{i\in I^h} \cC_i$.
\end{theorem}

\begin{proof}
The ``if'' part: Let $t \in I^h$ be such that $x_{t}\neq \bzero$. Then, by the definition of $\hat{x}$ we have
\[
x=\hat{x} \oplus \left(\bigoplus_{i\in \cup_{r>h} I^r} x_i \uvector{i}\right)\oplus
\left(\bigoplus_{j\in K^h} x_t (\uvector{t}\oplus x_t^{-1} x_j  \uvector{j})\right) \; .
\]
It follows that $x\in \cV$ because $\hat{x}\in\mysup_{i\in I^h} \cC_i\subseteq\cV$,
$\uvector{i}\in\cV$ for all $i\in I$ and $\uvector{i}\oplus \lambda \uvector{j}\in\cV$
for all $i\in I^h$, $j\in K^h$ and $\lambda\in \Rinv $.

The ``only if'' part: Let $x\in \cV$.
As $\cV=\mysup_{i\in I}(\cC_i\oplus\cD_i)$,
we have $x=\mysup_{i\in I} (y^i\oplus z^i)$ for some $y^i\in\cC_i$ and
$z^i\in\cD_i$. Note that $y^i\oplus z^i=\bzero$ for $i\in I^r$ with $r<h$ since
$y_i^i\oplus z_i^i=x_i=\bzero$ for such vectors. So
$x=\mysup_{i\in\cup_{r\geq h} I^r} (y^i\oplus z^i)$.

We will show that $y^i$ can be chosen so that
$\hat{x}=\mysup_{i\in I^h} y^i\in\mysup_{i\in I^h}\cC_i$.
For this, observe that for all $i\in I^h$, since $\uvector{i}\in\cC_i$,
we can assume $x_i=\hat{x}_i=y^i_i$,
adding $x_i\uvector{i}$ to $y^i$ if necessary. This fixes our choice of $y^i$.
Then by~\eqref{PropertyJrcupKrsubsetKr-1},
for $r> h$ we have $J^r\cup K^r\subseteq K^h$,
or equivalently, $J^h\cup L^h\subseteq L^r$.
It follows from~\eqref{def:CiDi} and the above that $\supp (y^i\oplus z^i) \subseteq  I^r\cup K^r\cup J^r =  I^r\cup (J\setminus L^r)
\subseteq [n]\setminus (I^h\cup J^h\cup L^h)$ for $i\in I^r$ and $r>h$. Thus, $\hat{x}_k=x_k=\left(\bigoplus_{i\in I^h} (y^i\oplus z^i)\right)_k$ for all
$k\in I^h\cup J^h\cup L^h$. Moreover, since we have $\supp (z^i)\subseteq K^h\cup\{i\}$ (from~\eqref{def:CiDi}) and $x_i=\hat{x}_i=y_i^i$ for $i\in I^h$, it follows that $\hat{x}_k=x_k=\left(\bigoplus_{i\in I^h} y^i\right)_k$ for all
$k\in I^h\cup J^h\cup L^h$. Finally,
the claim follows from the fact that $\hat{x}_k=\bzero =\left(\bigoplus_{i\in I^h} y^i\right)_k$ for
$k\not \in I^h\cup J^h\cup L^h$.
\end{proof}

We now describe $\mysup_{i\in I^r}\cC_i$ as the set of vectors lying in a
halfspace~\eqref{e:Ci-cond1} and satisfying a constraint~\eqref{e:Ci-cond2}.

\begin{lemma}\label{l:Ci-halfs}
If $J^r\neq\emptyset$, then the non-null elements of the cone
$\mysup_{i\in I^r}\cC_i$ are the vectors $x\in\Rmax^n$ that satisfy
$x_i\neq\bzero$ for some $i\in I^r$,
\begin{equation}\label{e:Ci-cond1}
\mysup_{j\in J^r} \gamma_j x_j \leq \mysup_{i\in I^r} \beta_i x_i\makebox{ and }   x_j=\bzero \makebox{ for } j\notin I^r\cup J^r,
\end{equation}
and, in addition,
\begin{equation}\label{e:Ci-cond2}
\gamma_j x_j =\mysup_{i\in I^r} \beta_i x_i \implies
\exists k \in I^r \makebox{ such that } \gamma_j x_j=\beta_{k} x_{k} \makebox{ and } j\in J_{k}^{\leq}.
\end{equation}
\end{lemma}

\begin{proof}
Assume first that the conditions are satisfied for $x\in \Rmax^n$.
Given $j\in J^r$, if $\gamma_j x_j = \bigoplus_{i\in I^r} \beta_i
x_i$, let ${k}\in I^r$ be such that $\beta_{k} x_{k}=\bigoplus_{i\in
I^r} \beta_i x_i$ and $j\in J_{k}^{\leq}$. Then, the vector $y^{k
j}:=\uvector{k}\oplus x_j x_{k}^{-1} \uvector{j}$ belongs to
$\cC_{k}$ because $j\in J_{k}^{\leq}$ and $x_j
x_{k}^{-1}=\beta_{k}\gamma_j^{-1}=\sigma_{k j}$. Given $j\in J^r$
such that $\gamma_{j} x_j < \bigoplus_{i\in I^r} \beta_i x_i$, let
$k$ be any element of $I^r$ such that $\beta_{k} x_{k}$ attains the
maximum in $\bigoplus_{i\in I^r} \beta_i x_i$. The vector $y^{k
j}:=\uvector{k}\oplus x_j x_{k}^{-1} \uvector{j}$ again belongs to
$\cC_{k}$, because $j\in J_{k}^{\leq}\cup J_{k}^<$ and $x_j
x_{k}^{-1}<\beta_{k}\gamma_j^{-1}=\sigma_{k j}$. Since
$\uvector{i}\in \cC_i$ for all $i\in I^r$, it readily follows that
$x\in \mysup_{i\in I^r}\cC_i$ as a sum of $x_i\uvector{i}$ for $i\in
I^r$ and $x_{k} y^{k j}=x_{k} \uvector{k}\oplus x_j\uvector{j}$ over
all $y^{k j}$ considered above.

Assume now that $x\in\mysup_{i\in I^r}\cC_i$ is non-null.
Represent $x= \mysup_{i\in I^r}y^i$ where $y^i\in\cC_i$.
Using~\eqref{e:Ci} we observe that each vector $y$ in $\cC_i$ for $i\in I^r$ satisfies
$\mysup_{j\in J^r} \gamma_j y_j\leq \beta_i y_i$ and $y_h=\bzero$ for all $h\notin I^r\cup J^r$,
hence it lies in the halfspace~\eqref{e:Ci-cond1}, and so the same holds for $x$.
Besides, the fact that $x\neq \bzero$ and~\eqref{e:Ci-cond1} imply that
$x_i\neq\bzero$ for some $i\in I^r$.
Finally, if $\gamma_{j} x_j = \mysup_{i\in I^r} \beta_i x_i$,
let $k \in I^r$ be such that $x_j=y^{k}_j$. Since $y^{k}\in \cC_{k}$,
by~\eqref{e:Ci} we have $\gamma_{j} y^{k}_j\leq \beta_{k} y^{k}_{k}$, and
it follows that $\gamma_{j} x_j = \gamma_{j} y^{k}_j\leq \beta_{k} y^{k}_{k} \leq \beta_{k} x_{k} \leq \mysup_{i\in I^r} \beta_i x_i$.
All these inequalities turn into equalities, so we have
$\gamma_{j} y^{k}_j= \beta_{k} y^{k}_{k}$ with $y^{k}\in\cC_{k}$,
and hence $j\in J^{\leq}_{k}$ by~\eqref{e:Ci}. This shows that the conditions of the lemma are also necessary.
\end{proof}

\begin{proposition}\label{p:Ci-hemi}
For each $r\in [p]$ the cone $\mysup_{i\in I^r} \cC_i$ is a conical
hemispace.
\end{proposition}

\begin{proof}
The case when $J^r=\emptyset$ was treated in~\eqref{e:hemitriv}, so
we can assume $J^r\neq\emptyset$. We have shown that the non-trivial
elements of $\mysup_{i\in I^r} \cC_i$ are precisely the elements of
$\Rmax^n$ that satisfy~\eqref{e:Ci-cond1} and~\eqref{e:Ci-cond2}. In
the rest of the proof, we assume that the complement of $I^r\cup
J^r$ is empty, or equivalently, we will show that $\mysup_{i\in I^r}
\cC_i$ is a conical hemispace in the plane $\{x_i=\bzero\mid i\notin
I^r\cup J^r\}$, from which it follows that $\mysup_{i\in I^r} \cC_i$
is a conical hemispace in $\Rmax^n$. (For this, verify that the
complement of a cone lying in $\{x_i=\bzero\mid i\in \Tilde{I}\}$,
for $\Tilde{I}$ a subset of $[n]$, is a cone, if the restriction of
that complement to $\{x_i=\bzero\mid i\in \Tilde{I}\}$ is a cone.)
Thus, we assume $I^r\cup J^r=[n]$.

Let us build a ``reflection'' of $\mysup_{i\in I^r} \cC_i$,
swapping the roles of $I^r$ and $J^r$,
and the roles of $J_k^{\leq}$ and $J_k^<$
in~\eqref{e:Ci-cond1} and~\eqref{e:Ci-cond2}.
Namely, we define it as the set $\Tilde{\cC}$
containing $\bzero$ and all the vectors $x\in\Rmax^n$ that satisfy
\begin{equation}\label{e:Ci-cond1r}
\mysup_{i\in I^r} \beta_i x_i\leq \mysup_{j\in J^r} \gamma_j x_j
\end{equation}
and
\begin{equation}\label{e:Ci-cond2r}
\beta_ix_i=\mysup_{j\in J^r} \gamma_j x_j \implies
\exists k\in J^r \makebox{ such that } \gamma_k x_k=\beta_i x_i
\makebox{ and }  k\in J_i^<.
\end{equation}

We need to show that $\Tilde{\cC}$ is a cone.
Evidently, $x\in\Tilde{\cC}$ implies
$\lambda x\in\Tilde{\cC}$ for all $\lambda\in\Rmax$.
If $x,y\in\Tilde{\cC}\setminus \{\bzero \}$ and $z=x\oplus y$ satisfies~\eqref{e:Ci-cond1r}
with strict inequality, then $z\in\Tilde{\cC}$. If not,
let $i$ be such that $\beta_iz_i=\mysup_{j\in J^r} \gamma_j z_j$,
and assume $z_i=x_i$.
It follows that $\beta_ix_i=\mysup_{j\in J^r} \gamma_j x_j$,
and then there exists $k\in J^r$ such that
$\gamma_k x_k=\beta_i x_i$ and $k\in J_i^<$.
Further observe that
$\gamma_k z_k\geq \gamma_k x_k=\beta_ix_i=\beta_iz_i=\mysup_{j\in J^r} \gamma_j z_j\geq \gamma_k z_k$, and so $\gamma_k z_k=\beta_i z_i$,
showing that $z$ satisfies~\eqref{e:Ci-cond2r} and is in $\Tilde{\cC}$.

We now show that $\Tilde{C}\setminus \{\bzero \}$ is the complement
of $\mysup_{i\in I^r} \cC_i$, so $\Tilde{C}$ and $\mysup_{i\in I^r}
\cC_i$ form a joined pair of conical hemispaces. Building the
complement of $\mysup_{i\in I^r} \cC_i$ by
negating~\eqref{e:Ci-cond1} and~\eqref{e:Ci-cond2}, we see that it
consists of two branches: vectors $x$ satisfying
\[
\mysup_{i\in I^r} \beta_i x_i< \mysup_{j\in J^r} \gamma_j x_j,
\]
and those satisfying
$$\mysup_{i\in I^r} \beta_i x_i = \mysup_{j\in J^r} \gamma_j x_j $$
and
\[
\exists k\in J^r \makebox{ such that }
\gamma_k x_k = \mysup_{i\in I^r} \beta_i x_i,
\makebox{ and } k\in J_{h}^<
\makebox{ whenever }
\beta_{h}x_{h}=\gamma_k x_k.
\]
It can be verified that both branches belong to the ``reflection'' $\Tilde{\cC}$ as defined
by~\eqref{e:Ci-cond1r} and~\eqref{e:Ci-cond2r}.

We are now left to show that $\mysup_{i\in I^r} \cC_i$ and its ``reflection'' $\Tilde{\cC}$ do not contain
any common non-null vector. We will use~\eqref{PropertyChain}, i.e.,
the fact that for each $i_1,i_2\in I^r$ either
$J_{i_1}^<\subseteq J_{i_2}^<$ or $J_{i_2}^<\subseteq J_{i_1}^<$.
This property means that the sets $J_i^<$ and $J_i^{\leq}=J\bez J_i^{<}$ are nested, hence the elements of $I^r$ and $J^r$ can
be assumed to be ordered so that
\begin{equation*}
i_1\le i_2\Leftrightarrow J_{i_2}^{\leq}\subseteq J_{i_1}^{\leq}
\end{equation*}
and the following properties are satisfied:
\begin{equation}
\label{e:nest-order}
\begin{split}
j_1 \in J^{\leq}_{i_1}\; ,\;  j_2 \in J^{<}_{i_1}  &\implies  \; j_1 < j_2\; ,\\
j_1 \in J^{<}_{i_1}\; , \; j_1 \in J^{\leq}_{i_2} & \implies \; i_2< i_1\; .
\end{split}
\end{equation}

Assume now $x\in \left(\mysup_{i\in I^r}
\cC_i\right)\cap\Tilde{\cC}$ but $x\neq \bzero$. Then, we
necessarily have $\mysup_{i\in I^r} \beta_i x_i = \mysup_{j\in J^r}
\gamma_{j} x_j\neq \bzero$. Let $i_1\in I^r$ be such that
$\beta_{i_1} x_{i_1}=\mysup_{j\in J^r} \gamma_j x_j $. Since $x \in
\Tilde{\cC}$, there exists $j_1\in J_{i_1}^{<}$ such that
$\mysup_{j\in J^r} \gamma_{j} x_j=\gamma_{j_1} x_{j_1}$. As $x \in
\mysup_{i\in I^r} \cC_i$, there exists $i_2\in I^r$ such that
$\beta_{i_2} x_{i_2}=\mysup_{i\in I^r} \beta_i x_i=\gamma_{j_1}
x_{j_1} $ and $j_1\in J_{i_2}^{\leq}$, and so $i_2 < i_1$
by~\eqref{e:nest-order}. Again, using the fact that $x \in
\Tilde{\cC}$ and $\beta_{i_2} x_{i_2}=\mysup_{j\in J^r} \gamma_j
x_j$, we conclude that there exists $j_2\in J_{i_2}^{<}$ such that
$\mysup_{j\in J^r} \gamma_{j} x_j=\gamma_{j_2} x_{j_2}$, and so $j_1
< j_2$ by~\eqref{e:nest-order}. Repeating this argument again and
again we obtain infinite sequences $i_1>i_2>i_3>\ldots$ and
$j_1<j_2<j_3<\ldots$,  which is impossible. Hence, $\mysup_{i\in
I^r} \cC_i$ and $\Tilde{\cC}$ form a joined pair of conical
hemispaces.
\end{proof}

\begin{remark}\label{JustRemark}
It can be shown that $\Tilde{\cC}=\mysup_{j\in J^r} \Tilde{\cC}_j$,
where $\Tilde{\cC}_j$ are defined as the ``reflection'' of $\cC_i$,
i.e., cones whose non-null vectors satisfy
\[
\left\{
\begin{array}{l}
\beta_i x_i\leq \gamma_j x_j \makebox{ for all } i \makebox{ such that }
 j\in J_i^{<}\\
\beta_i x_i< \gamma_j x_j 
\makebox{ for all }  i \makebox{ such that }  j\in J_i^{\leq} \\
x_i=\bzero \makebox{ for all }  i\in J^r\setminus \{j\}
\end{array}
\right.
\]
The proof of $\Tilde{\cC}=\mysup_{j\in J^r} \Tilde{\cC}_j$ is based on the
arguments of Lemmas~\ref{l:Ci} and~\ref{l:Ci-halfs}.
As this observation is just a remark, we will not provide a proof.
\end{remark}

\begin{proof}[Proof of the ``if'' part of Theorem~\ref{t:Ghemi}]

Let $\cC_i\subset \Rmax^n$, for $i\in I$,
be defined by~\eqref{def:CiDi} (see also~\eqref{e:Ci},
a working equivalent definition,
and~Lemma~\ref{l:Ci-halfs} for an equivalent definition of
$\mysup_{i\in I^r} \cC_i$).
Let the operator $x\mapsto \hat{x}$
be defined as in Theorem~\ref{t:hatx}.

Let $x\in \complement \cV$ (which in particular means $x\neq \bzero$) and $\lambda \in \Rinv$.
If $x_i = \bzero$ for all $i\in I$, then $\lambda x \in \complement \cV$ is immediate by
Remark~\ref{RemarkNullVectorOnly} because $x\neq \bzero$.
If $x_i\neq \bzero$ for some $i\in I$, let
$h:=\min \{r\in [p]\mid x_t \neq \bzero \makebox{ for some } t \in I^r \}$.
Then, $\hat{x}\not \in \mysup_{i\in I^h} \cC_i$ by
Theorem~\ref{t:hatx} because $x\in \complement \cV$.
Note that for $y:=\lambda x$ we have
$\min \{r\in [p]\mid y_t \neq \bzero \makebox{ for some } t \in I^r \}=h$
and $\hat{y} =\lambda \hat{x}$. By Theorem~\ref{t:hatx}
it follows that $y\in \complement \cV$
because $\hat{y} =\lambda \hat{x} \not \in \mysup_{i\in I^h} \cC_i$.

Let now $x,y \in \complement \cV$
(which in particular means $x\neq \bzero$ and $y\neq \bzero$)
and define $z:=x\oplus y$.

Assume first that $x_i=y_i=  \bzero$ for all $i\in I$.
Then, $z_i=\bzero$ for all $i\in I$, and as $z\neq \bzero$,
we conclude $z\in \complement \cV$ by Remark~\ref{RemarkNullVectorOnly}.

In the second place, assume $x_i\neq \bzero$ for some $i\in I$
but $y_t = \bzero$ for all $t \in I$. Then,
note that $\hat{z}=\hat{x}\oplus w$ for some vector $w$
which satisfies $\supp (w)\cap I=\emptyset$.
Let $h:=\min \{r\in [p]\mid x_t \neq \bzero \makebox{ for some } t \in I^r \}$,
so $\hat{x} \not \in \mysup_{i\in I^h} \cC_i$ by Theorem~\ref{t:hatx}.
Since $\hat{z}=\hat{x}\oplus w$ and $\supp (w)\cap I^h=\emptyset$,
from Lemma~\ref{l:Ci-halfs} it follows that
$\hat{z}\not \in \mysup_{i\in I^h} \cC_i$,
and so $z\in  \complement \cV$ by Theorem~\ref{t:hatx}.

Finally, assume $x_i\neq \bzero$ and $y_t \neq \bzero$ for some $i,t
\in I$. Let $h:=\min \{r\in [p]\mid x_t\neq \bzero \makebox{ for
some } t \in I^r \}$ and $k:=\min \{r\in [p]\mid y_t \neq \bzero
\makebox{ for some } t \in I^r \}$. We first consider the case $h
\neq k$, and so without loss of generality we may assume $h < k$.
Then, as above, we conclude that $z \in \complement \cV$ because
$\hat{z}=\hat{x}\oplus w$ for some vector $w$ satisfying $\supp
(w)\cap I^h=\emptyset$. Suppose now $h=k$. Then, $\min \{r\in
[p]\mid z_t \neq \bzero \makebox{ for some } t \in I^r \}=h$ and
$\hat{z}=\hat{x}\oplus \hat{y}$. From $\hat{x}\not \in \mysup_{i\in
I^h} \cC_i$ and $\hat{y}\not \in \mysup_{i\in I^h} \cC_i$, it
follows that $\hat{z}\not \in \mysup_{i\in I^h} \cC_i$, because
$\mysup_{i\in I^h} \cC_i$ is a conical hemispace by
Proposition~\ref{p:Ci-hemi}. Thus, again by Theorem~\ref{t:hatx}, we
have $z \in \complement \cV$.
\end{proof}

\begin{example}
Let us consider the cone
\[
\cV =\cone \left(\left\{\uvector{1}\right\} \cup \left\{\uvector{1}\oplus \uvector{3}\right\} \cup \left\{\uvector{1}\oplus \delta \uvector{4}\mid \delta \in \Rmax \right\} \cup \left\{\uvector{2}\right\} \cup \left\{ \uvector{2}\oplus \uvector{4}\right\}\right) \subseteq \Rmax^4 \; .
\]
Note the $\cV$ can be written in the form~\eqref{DefHemiWithGen} defining $I:=\{1,2\}$, $J:=\{3,4\}$,
$\sigma^{(-)}_{13}:=\{\lambda \mid \lambda \leq \bunity\}$, $\sigma^{(-)}_{14}:= \Rmax$, $\sigma^{(-)}_{23}:=\{\bzero \}$ and $\sigma^{(-)}_{24}:=\{\lambda \mid \lambda \leq \bunity\}$. Since the rank-one condition~\eqref{e:rankone} is satisfied with $\sigma^{(+)}_{13}:=\Rmax \cup \{+\infty \}\setminus \sigma^{(-)}_{13} = \{\lambda \mid \lambda > \bunity\}$, $\sigma^{(+)}_{14}:=\Rmax \cup \{+\infty \}\setminus \sigma^{(-)}_{14}= \{+\infty \}$, $\sigma^{(+)}_{23}:=\Rmax \cup \{+\infty \}\setminus\sigma^{(-)}_{23}=\Rinv \cup \{+\infty \}$ and $\sigma^{(+)}_{24}:=\Rmax \cup \{+\infty \}\setminus \sigma^{(-)}_{24}=\{\lambda \mid \lambda > \bunity\}$, by Theorem~\ref{t:Ghemi} we know that $\cV$ is a conical hemispace. Then, by
Proposition~\ref{p:complement} we also know that $\cV_1:=\cV$ and
\[
\cV_2:=\cone \left( \left\{\uvector{3}\right\} \cup \left\{\uvector{3}\oplus \alpha \uvector{1}\mid \alpha<\bunity \right\} \cup \left\{\uvector{3}\oplus \beta \uvector{2}\mid \beta \in \Rmax \right\} \cup \left\{\uvector{4}\right\} \cup\left\{\uvector{4}\oplus \gamma \uvector{2} \mid \gamma< \bunity \right\} \right)
\]
form a joined pair of conical hemispaces. Let us verify that this holds.

We first show that $\cV_1\cap \cV_2=\{\bzero\}$.
Assume $x\in \cV_1\cap \cV_2$. Note that we can always
express $x$ as a linear combination of the generators of $\cV_1$
containing at most one vector of the form
$\uvector{1} \oplus \delta \uvector{4}$.
The same observation holds for the generators of $\cV_2$
and vectors of the form
$\uvector{3} \oplus \alpha \uvector{1}$,
$\uvector{3} \oplus \beta \uvector{2}$ and
$\uvector{4} \oplus \gamma \uvector{2}$. Thus, we have
\[
x = \mu_1 \uvector{1}\oplus \mu_2(\uvector{1}\oplus \uvector{3})\oplus \mu_3(\uvector{1}\oplus \delta \uvector{4})\oplus \mu_4\uvector{2}\oplus \mu_5(\uvector{2}\oplus \uvector{4})
\]
for some $\mu_1,\mu_2,\mu_3,\mu_4,\mu_5\in \Rmax$ since $x\in \cV_1$, and
\[
x = \nu_1 \uvector{3}\oplus \nu_2(\uvector{3}\oplus\alpha \uvector{1})\oplus \nu_3(\uvector{3}\oplus \beta \uvector{2})\oplus \nu_4\uvector{4}\oplus \nu_5 (\uvector{4}\oplus \gamma \uvector{2})
\]
for some $\nu_1,\nu_2,\nu_3,\nu_4,\nu_5\in \Rmax$ since $x\in \cV_2$.

Writing the equality on components in these expressions gives:
\begin{equation}\label{eq:comp}
\begin{aligned}
\mu_1 \oplus\mu_2 \oplus\mu_3 &=\alpha \nu_2, \\
\mu_4\oplus \mu_5 &= \nu_3 \beta\oplus \nu_5\gamma,\\
\mu_2&=\nu_1 \oplus \nu_2 \oplus \nu_3,\\
\mu_3 \delta \oplus \mu_5 &=\nu_4\oplus \nu_5.
\end{aligned}
\end{equation}

From the first and third equalities in \eqref{eq:comp} it follows that
$$
\mu_2\le \mu_1\oplus \mu_2 \oplus \mu_3 =\alpha \nu_2\le \alpha (\nu_1 \oplus \nu_2 \oplus \nu_3)=\alpha \mu_2,
$$
which, due to $\alpha <\bunity $, implies $\mu_1=\mu_2=\mu_3=\nu_1=\nu_2=\nu_3=\bzero$.
Then, from the second and fourth equalities in \eqref{eq:comp} it follows that
$$
\mu_5 \leq \mu_4\oplus\mu_5 =\nu_5 \gamma\le (\nu_4\oplus\nu_5)\gamma =\mu_5 \gamma,
$$
which, due to $\gamma <\bunity $, implies $\mu_4=\mu_5=\nu_4=\nu_5=\bzero$.

To show that $\cV_1\cup \cV_2=\Rmax^4$, let $x\in \Rmax^4$.
It is convenient to consider different cases.

If $x_1=x_3=\bzero$, we have $x=x_4(\uvector{2}\oplus \uvector{4}) \oplus x_2 \uvector{2}\in \cV_1$ when $x_2\geq x_4$, and defining $\gamma =x_4^{-1} x_2$
we have $x= x_4 (\uvector{4}\oplus \gamma \uvector{2})\in \cV_2$
when $x_2< x_4$.

When $x_1=\bzero$ and $x_3\neq \bzero$, defining $\beta =x_3^{-1} x_2$ we have
$x= x_4 \uvector{4} \oplus x_3 (\uvector{3}\oplus \beta \uvector{2})\in \cV_2$.

When $x_1\neq \bzero$ and $x_3=\bzero$, defining $\delta =x_1^{-1} x_4$ we have
$x= x_2 \uvector{2} \oplus x_1 (\uvector{1}\oplus \delta \uvector{4})\in \cV_1$.

If $x_1\neq \bzero$ and $x_3\neq \bzero$, defining
$\delta =x_1^{-1} x_4$ we have $x=x_1 \uvector{1} \oplus x_2 \uvector{2} \oplus x_3(\uvector{1}\oplus \uvector{3})\oplus  x_1 (\uvector{1}\oplus \delta \uvector{4})\in \cV_1$ when $x_1\geq x_3$, and defining $\beta =x_3^{-1} x_2$ and $\alpha =x_3^{-1}x_1$ we have
$x=x_3 \uvector{3} \oplus x_4 \uvector{4} \oplus x_3 (\uvector{3}\oplus
\beta \uvector{2}) \oplus x_3 (\uvector{3} \oplus \alpha \uvector{1})\in \cV_2$ when $x_1 < x_3$.
\end{example}

\subsection{Closed hemispaces and closed halfspaces}
\label{ss:final}

We now consider the case of closed conical hemispaces, and show that
these are precisely the closed homogeneous halfspaces, i.e., cones
of the form
\begin{equation}\label{e:halfs}
\left\{x\in \Rmax^n \mid \mysup_{j\in J}\gamma_j x_j\leq\mysup_{i\in I} \beta_i x_i\;\makebox{and}\;  x_i=\bzero\;\makebox{for all}\; i\in L \right\},
\end{equation}
where $I$, $J$ and $L$ (with $I$ and $J$, or $L$, possibly empty)
are pairwise disjoint subsets of $[n]$.

\begin{theorem}[Briec and Horvath~\cite{BH-08}]\label{t:hemishalfs}
Closed conical hemispaces $=$ closed homogeneous halfspaces.
\end{theorem}

\begin{proof}
Closed homogeneous halfspaces are closed conical hemispaces, since
the complement of~\eqref{e:halfs} is given by
\[
\left\{x\in \Rmax^n \mid \mysup_{j\in J}\gamma_j x_j > \mysup_{i\in I} \beta_i x_i\;\makebox{or}\;  x_i\neq \bzero\;\makebox{for some}\; i\in L \right\},
\]
and adding the null vector $\bzero$ to this complement we get a cone.

Conversely, if a conical hemispace $\cV$ is closed, then
in~\eqref{DefHemiWithGen} we have $\sigma_{ij} \in \Rmax$ for all
$i\in I$ and $j\in J$, and the sets $\sigma_{ij}^{(-)}$ can only be of the form
\[
\sigma_{ij}^{(-)}=
\begin{cases}
\{\lambda \mid \lambda \leq \sigma_{ij}\} &
\text{if } \sigma_{ij} \in \Rinv \; , \\
\{\sigma_{ij}\} & \text{if } \sigma_{ij} = \bzero \; .
\end{cases}
\]
Equivalently, the sets $J_i^<$ and $J_i^{\infty}$ of
Proposition~\ref{p:RD} are empty for all $i\in I$, and so
$K^r=\emptyset$ for $r\in [p]$. Observe that this means that $L^r=J$
if $J^r=\emptyset$, which in turn implies $p=r$. Moreover, we also
have $\cV=\mysup_{i\in I} (\cC_i\oplus \cD_i)= \mysup_{i\in I}
\cC_i$ if $\cV$ is a closed conical hemispace, since
$J_i^{\infty}=\emptyset$ implies $\cD_i \subseteq \cC_i$.

Assume first that $p\geq 2$, which implies $J^1\neq \emptyset $
as mentioned above. Then, we have
$J^2\cup K^2 \subseteq K^1=\emptyset$ by~\eqref{PropertyJrcupKrsubsetKr-1}.
It follows that $J^2=\emptyset$, and so $p=2$.
Thus, we have $I=I^1\cup I^2$ and
$\cV=\mysup_{i\in I^1\cup I^2} \cC_i$.
By Lemma~\ref{l:Ci-halfs},
the cone $\bigoplus_{i\in I^1}\cC_i$ can be represented by
\begin{equation}\label{e:halfs-emerge}
\mysup_{j\in J^1} \gamma_j x_j\leq\mysup_{i\in I^1} \beta_i x_i \makebox{ and } x_j=\bzero \makebox{ for } j\in L^1\cup I^2.
\end{equation}
Note that this is just condition~\eqref{e:Ci-cond1},
and condition~\eqref{e:Ci-cond2} is always satisfied as
$J_{k}^{\leq}=J^1$ for all $k \in I^1$. Since $J^2=\emptyset$,
it follows that $\mysup_{i\in I^2} \cC_i$ is generated by
$\{\uvector{i}\mid i\in I^2\}$, and then~\eqref{e:halfs-emerge}
implies that $\cV=\mysup_{i\in I^1\cup I^2} \cC_i$
is the set of all vectors satisfying
\begin{equation}\label{e:closed-hemi}
\mysup_{j\in J^1} \gamma_j x_j\leq\mysup_{i\in I^1} \beta_i x_i\makebox{ and } x_j=\bzero\;\makebox{for}\; j\in L^1,
\end{equation}
which is a closed homogeneous halfspace.
Note that by Lemma~\ref{l:Ci-halfs} we arrive at the same
conclusion if we assume that $p=1$ and $J^1\neq \emptyset$.

Finally, if we assume that $p=1$ and $J^1= \emptyset$, then
$\cV=\mysup_{i\in I^1} \cC_i$ is generated by $\{\uvector{i}\mid i\in I^1=I\}$,
i.e., $\cV=\{x\in \Rmax^n\mid x_j=\bzero \makebox{ for } j\in J\}$
is a closed homogeneous halfspace.
\end{proof}

We now recall an important observation of~\cite{BH-08},
which will allow us to easily extend the result
of Theorem~\ref{t:hemishalfs} to general hemispaces.
For the reader's convenience, we give an elementary proof
based on (tropical) segments and their perturbations.

\begin{lemma}[Briec and Horvath~\cite{BH-08}]\label{LemmaBandH}
Closures of hemispaces $=$ closed hemispaces.
\end{lemma}
\begin{proof}[Proof (in the max-times setting, with usual arithmetics)]
Consider the closure of a hemispace $\cH$ in $\Rmaxtimes^n$.
Since the closure of a convex set is a closed convex set
(see e.g.~\cite{GK-07,BSS}), we only need to show that the complement
of this closure is also convex.
This complement is open, so it consists of all points
$x\in\complement\cH$ for which there exists an open ``ball''
$B_x^{\epsilon }:=\{u\in \Rmaxtimes^n\mid |u_i-x_i|< \epsilon \makebox{ for all } i\in [n]\}$
such that $B_x^{\epsilon }\subseteq\complement\cH$.
We need to show that if $x$ and $y$ have this property,
then any linear combination $z=\lambda x\oplus \mu y$ with
$\lambda\oplus\mu=1$ also does. If we assume $\lambda=1$, then
\[
z_i=
\begin{cases}
\mu y_i, & \text {if }\mu y_i>x_i,\\
x_i, & \text{if }\mu y_i\leq x_i.
\end{cases}
\]
Let us consider $\hat{z}\in \Rmaxtimes^n$
defined by $\hat{z}_i:=z_i+\epsilon_i$,
where $\epsilon_i'$ are such that $|\epsilon_i| \leq \epsilon $
for all $i\in [n]$. We can write
\[
\hat{z}_i=
\begin{cases}
\mu y_i +\epsilon_i, &\text{if }\mu y_i+\epsilon_i>x_i\text{ and }x_i<\mu y_i ,\\
\mu y_i+\epsilon_i = x_i+\epsilon'_i, &\text{if }\mu y_i+\epsilon_i\leq x_i<\mu y_i,\\
x_i+\epsilon_i, &\text{if } \mu y_i\leq x_i+\epsilon_i\text{ and }\mu y_i\leq x_i,\\
x_i+\epsilon_i=\mu y_i+\epsilon'_i, & \text{if } x_i+\epsilon_i<\mu y_i\leq x_i,
\end{cases}
\]
where always $|\epsilon'_i|\leq|\epsilon_i|\leq \epsilon $. Thus, defining
\[
\begin{cases}
\hat{y}_i:=y_i +\mu^{-1} \epsilon_i\text{ and } \hat{x}_i:=x_i, &\text{if }\mu y_i+\epsilon_i>x_i\text{ and }x_i<\mu y_i ,\\
\hat{y}_i:=y_i +\mu^{-1} \epsilon_i\text{ and }
\hat{x}_i:=x_i+\epsilon'_i, &\text{if } \mu y_i+\epsilon_i\leq x_i<\mu y_i,\\
\hat{y}_i:=y_i \text{ and } \hat{x}_i:=x_i+\epsilon_i, &\text{if }\mu y_i\leq x_i+\epsilon_i\text{ and }\mu y_i\leq x_i,\\
\hat{y}_i:=y_i +\mu^{-1} \epsilon'_i\text{ and } \hat{x}_i:=x_i+\epsilon_i, &\text{if } x_i+\epsilon_i<\mu y_i\leq x_i,
\end{cases}
\]
we have $\hat{z}=\mu \hat{y}\oplus \hat{x}$, $\hat{x}\in B_x^\epsilon $ and
$\hat{y}\in B_y^{\epsilon''}$, where $\epsilon'':=\mu^{-1} \epsilon $.
Since $\complement\cH$ is convex, it follows that
$B_z^\epsilon \subseteq\complement\cH$ if
$B_x^{\epsilon''}\subseteq\complement\cH$
and $B_y^{\epsilon''}\subseteq\complement\cH$, proving the claim.
\end{proof}

\begin{corollary}[Briec and Horvath~\cite{BH-08}]\label{c:hemishalfs}
Closed hemispaces $=$ closed halfspaces.
\end{corollary}

\begin{proof}
We need to consider the case of a closed halfspace
that is not necessarily homogeneous,
and of a closed hemispace.
A general closed halfspace is a set of the form
\begin{equation}\label{e:halfs1}
\left\{ x\in\Rmax^n \mid \mysup_{j\in J} \gamma_jx_j\oplus \alpha\leq \mysup_{i\in I} \beta_ix_i\oplus\delta \makebox{ and }
x_j=\bzero \makebox{ for } j\in L \right\},
\end{equation}
where $I$, $J$ and $L$ are pairwise disjoint subsets of $[n]$. As in
the case of conical hemispaces, it can be argued that the complement
is convex too, so~\eqref{e:halfs1} describes a hemispace.

Conversely, by Theorem~\ref{t:hemihomog}, for a general hemispace
$\cH\subseteq\Rmax^n$ there exists a conical hemispace $\cV\subseteq
\Rmax^{n+1}$ such that $\cH=C_{\cV}^{\bunity}$. Even if $\cH$ is closed,
$\cV$ may be not closed in general. However, if
$\overline{\cV}$ is the closure of $\cV$, then the section
$C_{\overline{\cV}}^{\bunity}$ still coincides with $\cH$. Indeed,
for any $z=(x,\bunity)\in \overline{\cV}$ there exists a sequence
$\{z^k\}_{k\in \N}$ of vectors of $\cV$ such that $\lim_k z^k = z$.
Since $z_{n+1}=\bunity$  and, by Proposition~\ref{p:section},
$C_{\cV}^{\alpha}=\{\alpha x\mid x\in\cH\}$ for any non-null
$\alpha$,
we can assume that $z^k=(\lambda_k
x^k,\lambda_k)$ for some $\lambda_k\in \Rmax $ and $x^k\in \cH$. It
follows that $\lim_k \lambda_k =\bunity$ and $\lim_k x^k = x$. Thus,
$x\in \cH$ because $\cH$ is closed. Therefore, we conclude that
$C_{\overline{\cV}}^{\bunity}=C_{\cV}^{\bunity}=\cH$.

By Lemma~\ref{LemmaBandH} it follows that $\complement
\overline{\cV}$ is convex, and so $\complement \overline{\cV}\cup
\{\bzero \}$ and $\overline{\cV}$ form a joined pair of conical
hemispaces. Then, by Theorem~\ref{t:hemishalfs}, $\overline{\cV}$
can be expressed as a solution set to
\[
\mysup_{j\in J} \gamma_jx_j\oplus \alpha x_{n+1}\leq \mysup_{i\in I} \beta_ix_i\oplus\delta x_{n+1} \makebox{ and }
x_j=\bzero \makebox{ for } j\in L,
\]
for some disjoint subsets $I$, $J$ and $L$ of $[n]$.
The original hemispace in $\Rmax^n$ appears as a section of this closed
homogeneous halfspace by $x_{n+1}=\bunity$, and so it is of the
form~\eqref{e:halfs1}.
\end{proof}

\begin{corollary}\label{c:openhemishalfs}
Open hemispaces $=$ open halfspaces.
\end{corollary}
\begin{proof}
Open hemispaces and open halfspaces can be obtained
as complements of their closed ``partner''.
\end{proof}

\subsection{Characterization of hemispaces by means of $(P,R)$-decompositions}\label{ss:mainres}

We now characterize hemispaces by means of $(P,R)$-decompositions,
as foreseen by Theorem~\ref{t:hemihomog} and Theorem~\ref{t:Ghemi}.

\begin{theorem}\label{t:mainres}
 Let $\cH$ be a non-empty proper convex subset of $\Rmax^n$. Then, $\cH$ is a hemispace
if and only if there exist non-empty disjoint sets $I$ and $J$ satisfying $I+J=[n+1]$
and $n+1\in I$, and sets $\sigma_{ij}^{(-)}$,
which are non-empty proper subsets of $\Rmax\cup \{+\infty\}$ either of the form
$\{\lambda \in \Rmax \mid \lambda \leq\sigma_{ij}\}$ or
$\{\lambda \in \Rmax \mid \lambda <\sigma_{ij}\}$ with
$\sigma_{ij}\in \Rmax\cup \{+\infty\}$,
such that the pairs $(\sigma^{(-)}_{ij},\sigma^{(+)}_{ij})$,
with $\sigma^{(+)}_{ij}$ defined by $\sigma^{(+)}_{ij}:=(\Rmax\cup \{+\infty\})\setminus \sigma^{(-)}_{ij}$, satisfy the rank-one condition~\eqref{e:rankone} and
\begin{equation}\label{hemi-grrep1}
\cH=\conv \left( \left\{\lambda \uvector{j}\mid j\in J, \lambda \in\sigma_{n+1,j}^{(-)}\right\} \right) \oplus \cone \left( \left\{\uvector{i}\oplus \lambda \uvector{j}\mid i\in I\setminus \{n+1\}, j\in J, \lambda \in\sigma_{ij}^{(-)}\right\} \right)
\end{equation}
if $\bzero \in \cH$, and
\begin{equation}\label{hemi-grrep2}
\cH=\conv \left( \left\{ \lambda \uvector{j}\mid j\in J, \lambda \neq +\infty, \lambda \in\sigma_{n+1,j}^{(+)}\right\} \right) \oplus\cone \left( \left\{ \uvector{i}\oplus \lambda \uvector{j} \mid i\in I\setminus \{n+1\}, j\in J, \lambda \in \sigma_{ij}^{(+)}\right\} \right)
\end{equation}
otherwise.
Moreover, if $\cH$ is a hemispace given by
the right-hand side of~\eqref{hemi-grrep1}, then $\complement\cH$
is given by the right-hand side of~\eqref{hemi-grrep2}, and vice versa.
\end{theorem}

\begin{proof}
Sufficiency: Consider the cones
\begin{equation}\label{v1v2-grep}
\begin{split}
\cV_1 = \cone \left( \left\{ \uvector{n+1}\oplus \lambda \uvector{j}\mid j\in J, \lambda \in\sigma_{n+1,j}^{(-)}\right\} \right) \oplus
\cone \left( \left\{ \uvector{i}\oplus \lambda \uvector{j}\mid i\in I\setminus \{n+1\},
j\in J, \lambda \in\sigma_{ij}^{(-)}\right\} \right) \; ,\\
\cV_2 = \cone \left( \left\{ \uvector{n+1}\oplus \lambda \uvector{j}\mid j\in J,  \lambda \in\sigma_{n+1,j}^{(+)}\right\} \right) \oplus\cone \left( \left\{ \uvector{i}\oplus \lambda \uvector{j} \mid i\in I\setminus \{n+1\},
j\in J, \lambda \in\sigma_{ij}^{(+)}\right\} \right) \; .
\end{split}
\end{equation}
By Theorem~\ref{t:Ghemi} (the ``if'' part), $\cV_1$ is a conical
hemispace. Further, by Proposition~\ref{p:complement}, $\cV_1$ and
$\cV_2$ form a joined pair of conical hemispaces.
Then, from Lemma~\ref{l:CharaHemiWithSections} it follows
that $C_{\cV_1}^{\bunity}$ and $C_{\cV_2}^{\bunity}$ form a
complementary pair of hemispaces. Besides, by
Proposition~\ref{p:gen-hom} we have $\cH=C_{\cV_1}^{\bunity}$ if
$\bzero \in \cH$ and $\cH=C_{\cV_2}^{\bunity}$ otherwise. Thus,
$\cH$ is a hemispace.

Necessity: If $\cH$ is a hemispace, then $(\cH,\complement \cH )$ is
a non-trivial complementary pair of hemispaces. By
Theorem~\ref{t:hemihomog}, $\cH$ and $\complement \cH$ can be
represented as sections of some conical hemispaces $\cV_1$ and
$\cV_2$, which form a joined pair of conical hemispaces. Since
$(\cH,\complement \cH )$ is non-trivial, it follows that
$(\cV_1,\cV_2)$ is also non-trivial. By
Theorem~\ref{t:Ghemi} (the ``only if'' part) and Proposition~\ref{p:complement},
$\cV_1$ and $\cV_2$ must be as in~\eqref{v1v2-grep}. Then, since
$\uvector{n+1}\in \cV_1$, we have $\cH=C_{\cV_1}^{\bunity}$ if
$\bzero \in \cH$ and  $\cH=C_{\cV_2}^{\bunity}$ otherwise.
Consequently, using Proposition~\ref{p:gen-hom}, we see that
$\cH$ has a $(P,R)$-decomposition as in~\eqref{hemi-grrep1} and its
complement as in~\eqref{hemi-grrep2} if $\bzero \in \cH$.
Similarly, $\cH$ has a $(P,R)$-decomposition as
in~\eqref{hemi-grrep2} and its complement as
in~\eqref{hemi-grrep1} if $\bzero\notin\cH$.
\end{proof}

\section*{Acknowledgement} We are grateful to Ivan Singer for very careful reading and numerous
suggestions aimed at improving the clarity of presentation and polishing the proofs.
We also thank Charles Horvath for
useful discussions, and, together with Walter Briec, for sending the
full text of their work~\cite{BH-08}.


\begin{thebibliography}{10}

\bibitem{AGNS-11}
M.~Akian, S.~Gaubert, V.~Nitica, and I.~Singer.
\newblock Best approximation in max-plus semimodules.
\newblock {\em Linear Alg. Appl.}, 435:3261--3296, 2011.

\bibitem{BH-04}
W.~Briec and C.~Horvath.
\newblock $\mathbb{B}$-convexity.
\newblock {\em Optimization}, 53:103--127, 2004.

\bibitem{BH-08}
W.~Briec and C.~Horvath.
\newblock Halfspaces and {H}ahn-{B}anach like properties in
  $\mathbb{B}$-convexity and max-plus convexity.
\newblock {\em Pacific J.~Optim.}, 4(2):293--317, 2008.

\bibitem{BHR-05}
W.~Briec, C.~Horvath, and A.~Rubinov.
\newblock {Separation in $\mathbb{B}$-convexity}.
\newblock {\em Pacific J. Optim.}, 1:13--30, 2005.

\bibitem{BSS}
P.~Butkovi{\v{c}}, H.~Schneider, and S.~Sergeev.
\newblock Generators, extremals and bases of max cones.
\newblock {\em Linear Alg. Appl.}, 421:394--406, 2007.

\bibitem{CDQV-85}
G.~Cohen, D.~Dubois, J.~P. Quadrat, and M.~Viot.
\newblock A linear system theoretic view of discrete event processes and its
  use for performance evaluation in manufacturing.
\newblock {\em IEEE Trans. on Automatic Control}, AC--30:210--220, 1985.

\bibitem{cgq-1999}
G.~Cohen, S.~Gaubert, and J.~P. Quadrat.
\newblock Max-plus algebra and system theory: where we are and where to go now.
\newblock {\em Annual reviews in control}, 23:207--219, 1999.

\bibitem{CGQ-04}
G.~Cohen, S.~Gaubert, and J.~P. Quadrat.
\newblock Duality and separation theorems in idempotent semimodules.
\newblock {\em Linear Alg. Appl.}, 379:395--422, 2004.
\newblock E-print \arxiv{math.FA/0212294}.

\bibitem{CGQS-05}
G.~Cohen, S.~Gaubert, J.~P. Quadrat, and I.~Singer.
\newblock Max-plus convex sets and functions.
\newblock In G.~Litvinov and V.~Maslov, editors, {\em Idempotent Mathematics
  and Mathematical Physics}, volume 377 of {\em Contemporary Mathematics},
  pages 105--129. AMS, Providence, 2005.
\newblock E-print \arxiv{math.FA/0308166}.

\bibitem{DS-04}
M.~Develin and B.~Sturmfels.
\newblock Tropical convexity.
\newblock {\em Documenta Math.}, 9:1--27, 2004.
\newblock E-print \arxiv{math.MG/0308254}.

\bibitem{GK-06}
S.~Gaubert and R.~D. Katz.
\newblock Max-plus convex geometry.
\newblock In R.~A. Schmidt, editor, {\em Proceedings of the 9th International
  Conference on Relational Methods in Computer Science and 4th International
  Workshop on Applications of Kleene Algebra (RelMiCS/AKA 2006)}, volume 4136
  of {\em Lecture Notes in Comput. Sci.}, pages 192--206. Springer, 2006.

\bibitem{GK-07}
S.~Gaubert and R.~D. Katz.
\newblock The {M}inkowski theorem for max-plus convex sets.
\newblock {\em Linear Alg. Appl.}, 421(2-3):356--369, 2007.
\newblock E-print \arxiv{math.GM/0605078}.

\bibitem{GM-09}
S.~Gaubert and F.~Meunier.
\newblock Carath{\'{e}}odory, helly and the others in the max-plus world.
\newblock {\em Discrete and Computational Geometry}, 43(3):648--652, 2010.
\newblock E-print \arxiv{0804.1361}.

\bibitem{GS-08}
S.~Gaubert and S.~Sergeev.
\newblock Cyclic projectors and separation theorems in idempotent convex
  geometry.
\newblock {\em Journal of Math. Sci.}, 155(6):815--829, 2008.
\newblock E-print \arxiv{0706.3347}.

\bibitem{GGMT-10}
M.~A. Goberna, E.~Gonz\'{a}lez, J.~E. Mart{\'{\i}}nez-Legaz, and
M.~I. Todorov.
\newblock {M}otzkin decomposition of closed convex sets.
\newblock {\em J. Math. Anal. Appl.}, 364:209--221, 2010.

\bibitem{Jos-05}
M.~Joswig.
\newblock Tropical halfspaces.
\newblock In J.~E. Goodman, J.~Pach, and E.~Welzl, editors, {\em Combinatorial
  and computational geometry}, volume~52 of {\em MSRI publications}, pages
  409--432. Cambridge Univ. Press, 2005.
\newblock E-print \arxiv{math/0312068}.

\bibitem{KM:97}
V.~N. Kolokoltsov and V.~P. Maslov.
\newblock {\em Idempotent analysis and its applications}.
\newblock Kluwer Academic Pub., 1997.

\bibitem{Lassak-84}
M.~Lassak.
\newblock Convex half-spaces.
\newblock {\em Fund. Math.}, 120:7--13, 1984.

\bibitem{LMS-01}
G.~L. Litvinov, V.~P. Maslov, and G.~B. Shpiz.
\newblock Idempotent functional analysis: An algebraic approach.
\newblock {\em Math. Notes (Moscow)}, 69(5):696--729, 2001.
\newblock E-print \arxiv{math.FA/0009128}.

\bibitem{MLegSin-84}
J.~E. Mart\'{\i}nez-Legaz and I.~Singer.
\newblock The structure of hemispaces in $\mathbb{R}^n$.
\newblock {\em Linear Alg. Appl.}, 110:117--179, 1988.

\bibitem{multi-order-1991}
J.~E. Martinez-Legaz and I.~Singer.
\newblock Multi-order convexity.
\newblock In {\em Applied Geometry and Discrete Mathematics (V. Klee
  Festschrift; P. Gritzmann and B. Sturmfels, eds.)}, volume~4 of {\em DIMACS
  Ser. Discrete Math. Theoretical Computer Sci.}, pages 471--488. Amer. Math.
  Soc., Princeton, 1991.

\bibitem{MS:92}
V.~P. Maslov and S.~N. Samborski{\u{\i}}, editors.
\newblock {\em {Idempotent analysis}}, volume~13 of {\em Advances in Sov.
  Math.} AMS, 1992.

\bibitem{NS-07I}
V.~Nitica and I.~Singer.
\newblock Max-plus convex sets and max-plus semispaces. {I}.
\newblock {\em Optimization}, 56:171--205, 2007.

\bibitem{NS-07II}
V.~Nitica and I.~Singer.
\newblock Max-plus convex sets and max-plus semispaces. {II}.
\newblock {\em Optimization}, 56:293--303, 2007.

\bibitem{NS-07}
V.~Nitica and I.~Singer.
\newblock The structure of max-plus hyperplanes.
\newblock {\em Linear Alg. Appl.}, 426(2-3):382--414, 2007.

\bibitem{Ser-09-inLS}
S.~Sergeev.
\newblock Multiorder, {Kleene} stars and cyclic projectors in the geometry of
  max cones.
\newblock In G.~L. Litvinov and S.~N. Sergeev, editors, {\em Tropical and
  Idempotent Mathematics}, volume 495 of {\em Contemporary Mathematics}, pages
  317--342. AMS, Providence, 2009.
\newblock E-print \arxiv{0807.0921}.

\bibitem{Sin:97}
I.~Singer.
\newblock {\em Abstract convex analysis}.
\newblock Wiley, 1997.

\bibitem{VdV}
M.~L.~J. Van~de Vel.
\newblock {\em Theory of convex structures}.
\newblock North-Holland, 1993.

\bibitem{Zim-77}
K.~Zimmermann.
\newblock A general separation theorem in extremal algebras.
\newblock {\em Ekonom.-Mat. Obzor (Prague)}, 13:179--201, 1977.

\end{thebibliography}


\end{document}